%% file: preprint.tex
\documentclass[a4paper,11pt,reqno]{amsart}

\usepackage{a4wide}

\usepackage[breaklinks,bookmarks=false]{hyperref}
\hypersetup{colorlinks,linkcolor=blue,citecolor=blue,urlcolor=blue,plainpages=false,pdfwindowui=false}

\usepackage{amssymb}
\usepackage[foot]{amsaddr}
\usepackage{mathrsfs}

\input{general_commands}

\input{elastic_commands}
\input{special_commands}

\numberwithin{equation}{section}
\numberwithin{theorem}{section}

\title[Stability estimates for nearly incompressible elastodynamics]%
{%
Frequency-explicit stability estimates for time-harmonic elastodynamic
problems in nearly incompressible materials
}

\author{T. Chaumont-Frelet$^\star$ and S. Nicaise$^\dagger$}
\address{\vspace{-.5cm}}
\address{\noindent \tiny \textup{$^\star$Inria Univ. Lille and Laboratoire Paul Painlev\'e, 59655 Villeneuve-d'Ascq, France}}
\address{\noindent \tiny \textup{$^\dagger$Universit\'e Polytechnique Hauts-de-France, INSA Hauts-de-France, CERAMATHS-Laboratoire de Mat\'eriaux C\'eramiques et Math\'ematiques, F-59313 - Valenciennes Cedex 9 France}}
\begin{document}

\maketitle

\begin{abstract}
\input{abstract}

\vspace{.5cm}
\noindent
{\sc Keywords:}
\input{keywords}
\end{abstract}

\input{body}

\bibliographystyle{amsplain}
\bibliography{bibliography}

\input{appendix}

\end{document}

%% file: general_commands.tex
\usepackage{amssymb}
\usepackage{mathrsfs}

\renewcommand{\Re}{\operatorname{Re}}
\renewcommand{\Im}{\operatorname{Im}}

\newcommand{\supp}{\operatorname{supp}}

\newcommand{\eq}{:=}

\newcommand{\grad}{\boldsymbol \nabla}
\renewcommand{\div}{\grad \cdot}
\newcommand{\curl}{\grad \times}

\newcommand{\BC}{\boldsymbol C}

\newcommand{\BH}{\boldsymbol H}

\newcommand{\BL}{\boldsymbol L}

\newcommand{\bg}{\boldsymbol g}
\newcommand{\bh}{\boldsymbol h}

\newcommand{\bk}{\boldsymbol k}

\newcommand{\bn}{\boldsymbol n}
\newcommand{\bo}{\boldsymbol o}

\newcommand{\bv}{\boldsymbol v}
\newcommand{\bw}{\boldsymbol w}
\newcommand{\bx}{\boldsymbol x}
\newcommand{\by}{\boldsymbol y}

\newcommand{\CO}{\mathcal O}

\newcommand{\LB}{\mathscr B}
\newcommand{\LC}{\mathscr C}

\newcommand{\LF}{\mathscr F}

\newcommand{\LK}{\mathscr K}

\newcommand{\LM}{\mathscr M}

\newcommand{\LR}{\mathscr R}

\newcommand{\LV}{\mathscr V}

%% file: elastic_commands.tex
\newcommand{\tens}[1]{\underline{\boldsymbol #1}}
\newcommand{\vect}[1]{\boldsymbol #1}

\newcommand{\tgrad}{\tens{\nabla}}
\newcommand{\eps}{\tens{\varepsilon}}
\newcommand{\sig}{\tens{\sigma}}

\newcommand{\vel}{\vartheta}

\newcommand{\vels}{\vartheta_{\rm S,\min}}

\newcommand{\velsV}{\vartheta_{\rm S}}
\newcommand{\velpV}{\vartheta_{\rm P}}

%% file: special_commands.tex
\newcommand{\Crob}{\LC_{\rm rob}}

\newcommand{\Creg}{\LC_{\rm reg}}
\newcommand{\Cell}{\LC_{\rm ell}}

\newcommand{\ttau}{\tens{\tau}}
\newcommand{\txi}{\tens{\xi}}

\newcommand{\RRR}{\ell}
\newcommand{\MMM}{M}
\newcommand{\mmm}{m}

\newcommand{\bzero}{\bo}

\newcommand{\Ks}{k_{\rm S}}
\newcommand{\Kp}{k_{\rm P}}

\newcommand{\ks}{\kappa_{\rm S}}
\newcommand{\kp}{\kappa_{\rm P}}

\newcommand{\uuu}{\vect{u}}
\newcommand{\fff}{\vect{f}}

\newcommand{\AAA}{\tens{A}}
\newcommand{\BBB}{\tens{B}}
\newcommand{\al}{\tens{\alpha}}

\newcommand{\bxi}{\vect{\xi}}

\newcommand{\R}{\mathbb{R}}
\newcommand{\C}{\mathbb{C}}

\newcommand{\nn}{\boldsymbol n}

\newcommand{\GDiss}{\Gamma_{\rm Diss}}
\newcommand{\GDir}{\Gamma_{\rm Dir}}

\newtheorem{theorem}{Theorem}
\newtheorem{lemma}[theorem]{Lemma}
\newtheorem{proposition}[theorem]{Proposition}
\newtheorem{corollary}[theorem]{Corollary}

\newtheorem{assumption}[theorem]{Assumption}
\newtheorem{remark}[theorem]{Remark}

\renewcommand{\Re}{\operatorname{Re}}
\renewcommand{\Im}{\operatorname{Im}}

\newcommand{\tw}{\widetilde w}
\newcommand{\tbv}{\widetilde \bv}
\newcommand{\hbv}{\widehat   \bv}
\newcommand{\hbx}{\widehat   \bx}
\newcommand{\hCO}{\widehat \CO}

\newcommand{\hOmega}{\widehat \Omega}
\newcommand{\hGDiss}{\widehat \Gamma_{\rm Diss}}

%% file: abstract.tex
We consider time-harmonic elastodynamic problems in heterogeneous media.
We focus on scattering problems in the high-frequency regime and in
nearly incompressible media, where the the angular frequency $\omega$
and ratio of the Lam\'e parameters $\lambda/\mu$ may both be large. We
derive stability estimates controlling the norm of the solution by the
norm of the right-hand side up to a fully-explicit constant. Crucially,
under natural assumptions on the domain and coefficients, this constant
increases linearly with $\omega$ and is uniform in the ratio $\lambda/\mu$.

%% file: keywords.tex
Elastodynamics,
Helmholtz problems,
Nearly incompressible materials,
High-frequency scattering,
Stability estimates,
Resolvent estimates

%% file: body.tex
\section{Introduction}

The propagation of time-harmonic waves plays a pivotal role in a vast array of applications
in physics and engineering, and is mathematically modeled via so-called Helmholtz equations.
Obtaining stability estimates controlling the norm of the solution to a Helmholtz equation
by the norm of the right-hand side through a constant explicit in the frequency is then crucial,
as such estimates are instrumental in control theory
\cite{bey_heminna_loheac_2003a,lagnese_1983a,lions_1988b,martinez_1999a}
and in the design and analysis of numerical discretizations
\cite{%
barucq_chaumontfrelet_gout_2017a,%
bernkopf_chaumontfrelet_melenk_2023a,%
chaumontfrelet_ern_vohralik_2021a,%
chaumontfrelet_nicaise_2020a,%
ciarlet_2002a,%
lafontaine_spence_wunsch_2022a,%
melenk_sauter_2011a,%
peterseim_2016a}.

The simplest instance of a time-harmonic wave propagation problem is the modeling of small
acoustic waves in an inifinite uniform fluid modeled by the scalar Helmholtz equation.
In this case, given $f: \R^3 \to \C$ supported in a ball $B_\RRR$ of radius $\RRR$,
the solution $u: \R^3 \to \C$ satisfies
\begin{equation}
\label{eq_helmholtz_acoustic_intro}
-\frac{\omega^2}{\vel^2} u-\Delta u = f,
\end{equation}
where $\omega > 0$ is the frequency and $\vel > 0$ is the (constant) wavespeed,
together with the Sommerfeld's radiation condition at infinity. In this setting,
the sharp stability estimate
\begin{equation}
\label{eq_sharp_acoustic_intro}
k^2 \|u\|_{B_\RRR}
\leq
k \RRR \|f\|_{B_\RRR}
\end{equation}
holds true \cite{galkowski_spence_wunsch_2020a}, where $k \eq \omega/\vel$ is the wavenumber.
Using the method of Morawetz's multipliers
\cite{%
barucq_chaumontfrelet_gout_2017a,%
brown_gallistl_peterseim_2017a,%
chandlerwilde_monk_2008a,%
hetmaniuk_2007a,%
melenk_1995a,%
moiola_spence_2019a,%
morawetz_1961a,%
perthame_vega_1999a,%
spence_2014a}, the estimate in \eqref{eq_sharp_acoustic_intro} can
be generalized, up to a multiplicative constant, to acoustic Helmholtz problems set
in the exterior of a star-shaped obstacle and/or if variable coefficients with suitable
monotonicity conditions are included in~\eqref{eq_helmholtz_acoustic_intro}.

In this work, we focus on elastodynamic problems. For an infinite homogeneous solid charactized
by its (constant) density and Lam\'e paramteres $\rho,\mu,\lambda > 0$, given $\fff: \R^3 \to \C^3$,
the solution $\uuu: \R^3 \to \C^3$ satisfies
\begin{equation}
\label{eq_helmholtz_elastic_intro}
-\rho \omega^2 \uuu - \mu \Delta \uuu - (\lambda+\mu) \grad \div \uuu = \rho \fff,
\end{equation}
together with a radiation condition at infinity. By looking at solutions of
the form $\uuu = \curl \boldsymbol \psi$ or $\uuu = \grad q$, one easily sees that
\eqref{eq_helmholtz_elastic_intro} supports two kinds of waves travelling at speeds
$\velsV \eq \sqrt{\mu/\rho}$ and $\velpV \eq \sqrt{(\lambda+2\mu)/\rho}$. These are
respectively called shear and pressure waves, and have associated wavenumbers
$\Ks \eq \omega/\velsV$ and $\Kp \eq \omega/\velpV$. In the particular 
case of homogeneous media considered above, the fundamental solution
(see \cite{brown_gallistl_2023a} and Appendix \ref{appendix_fundamental_solution} below)
may be employed to show that
\begin{equation}
\label{eq_fundamental_intro}
\Ks^2 \|\uuu\|_{B_\RRR}
\leq
(4 + 17\Ks\RRR) \mu^{-1} \|\fff\|_{B_\RRR}.
\end{equation}
We make two crucial observations regarding \eqref{eq_fundamental_intro}.
First (i), the scaling of the stability constant as $\Ks\RRR \to +\infty$,
i.e. as $\omega \to +\infty$, is the same as in the acoustic case. Besides (ii),
the estimate is uniform in the limit $\lambda/\mu \to +\infty$. This is in a sense
to be expected, since $\Kp \to 0$ as $\lambda$ becomes large.

The goal of the present work is to analyze stability properties of elastodynamic
problems in the high-frequency regime and in nearly incompressible media, meaning
that both $\omega$ and $\lambda/\mu$ may be large. The estimate in \eqref{eq_fundamental_intro}
is only available in infinite homogeneous media, since it explicitely uses the fundamental
solution in its proof. In contrast, here, we handle problems set in the exterior of an
obstacle and/or with variable coefficients $\rho$, $\lambda$ and $\mu$. Specifically,
given a bounded domain $\Omega \subset \R^3$ and $\fff: \Omega \to \C^3$,
we consider the problem of finding $\uuu: \Omega \to \C^3$ such that
\begin{equation}
\label{eq_helmholtz_strong}
\left \{
\begin{array}{rcll}
-\omega^2 \rho \uuu - \div \sig(\uuu) &=& \rho \fff & \text{ in } \Omega,
\\
\uuu &=& \bzero & \text{ on } \GDir,
\\
\sig(\uuu)\bn - i\omega\AAA\uuu &=& \bzero & \text{ on } \GDiss,
\end{array}
\right .
\end{equation}
where $\sig(\uuu) \eq \mu \{\tgrad \uuu + (\tgrad \uuu)^T\} + \lambda \div \uuu \tens{I}$
is the stress tensor, and the coefficients $\rho,\mu,\lambda: \Omega \to \R$ are allowed
to vary. The Dirichlet boundary condition on $\GDir$ represent an impenetrable obstacle,
and instead of the Sommerfeld radiation condition, we consider a first-order absorbing
condition on $\GDiss$ with impedance coefficients $\AAA: \GDiss \to \R^{3 \times 3}$.
Considering a local absorbing condition is crucial in applications linked with observability
and controllability \cite{bey_heminna_loheac_2003a,lagnese_1983a,lions_1988b,martinez_1999a},
and it is also important in numerical applications where the Sommerfeld radiation condition
has to be approximated by an absorbing boundary condition to truncate the computation domain
\cite{%
bayliss_gunzburger_turkel_1982a,%
engquist_majda_1979a,%
galkowski_lafontaine_spence_2021a,%
goldstein_1981a,%
higdon_1991a}.

Our key result is that, under natural assumptions concerning the star-shapeness of
the boundaries $\GDir$ and $\GDiss$ and radial monotonicity of the coefficients,
the bound
\begin{equation*}
\omega^2 \|\uuu\|_{\rho,\Omega}
\leq
C_{\rm st} \frac{\omega\RRR}{\vels} \|\fff\|_{\rho,\Omega}
\end{equation*}
where $\RRR$ is the radius of $\Omega$ and $\vels$ is the minimal shear wavespeed,
holds true with a fully-computable constant $C_{\rm st}$ that is independent of
both $\omega$ and $\lambda_{\max}/\mu_{\min}$. Crucially, the frequency-scaling
is optimal, and the estimate is robust in nearly incompressible materials where
$\lambda_{\max}/\mu_{\min} \to +\infty$. The fact that $C_{\rm st}$ is fully
computable is also fundamental in the context of a posteriori error estimation
for numerical approximations \cite{chaumontfrelet_ern_vohralik_2021a}.

Morawetz's multiplier techniques have been previously used in the litterature
to derive stability estimates for \eqref{eq_helmholtz_strong}, in particular
in \cite{bey_heminna_loheac_2003a,cummings_feng_2006a,lagnese_1983a}.
The present work improves on these references as follows. First (i),
the result given in \cite{cummings_feng_2006a} are suboptimal in
the frequency scaling (quadratic instead of linear). Besides (ii),
\cite{bey_heminna_loheac_2003a} and \cite{cummings_feng_2006a} are
limited to homogeneous media. Finally (iii), none of these works analyses
the nearly incompressible aspect.

Stability estimates are also established in \cite{brown_gallistl_2023a}
based on the fundamental solution. We slightly revisit this result in
Appendix \ref{appendix_fundamental_solution} to establish \eqref{eq_fundamental_intro}
in homogeneous media. However, this approach cannot be easily extended to
heterogeneous media.

We finally mention that techniques based on semi-classical analysis have been
used in the past for scalar Helmholtz problems \cite{galkowski_spence_wunsch_2020a}.
These technique are often sucessful in geometries where the Morawetz's multiplier fail.
However, to the best of our knowledge, the framework of semi-classical analysis has never
been employed in the context of elastodynamics so far.

The remainder of this work is organized as follows. Section \ref{section_settings}
recaps the setting and key notation. In Section \ref{section_standard}, we state
and revisit standard integration by parts identities in the context of Morawetz's
multipliers. Section \ref{section_new} is dedicated to a set of new identities
which allows to avoid the use of Korn constants and  obtain optimal frequency scaling.
The main stability estimates are established in Sections \ref{section_simple_robin}
and \ref{section_general_robin}, and these abstract estimates are applied to various
examples in Section \ref{section_examples}. We slightly revisit earlier works involving
fundamental solutions in Appendix \ref{appendix_fundamental_solution}, and state some
regularity results needed in Section \ref{section_general_robin} in
Appendix~\ref{appendix_elliptic_regularity}.

\section{Settings}
\label{section_settings}

\subsection{Domain}
\label{section_domain}

Throughout this work $\Omega \subset \R^d$, $d = 2$ or $3$, is a bounded domain with a Lipschitz
boundary. The boundary of $\Omega$ is partitioned into two non-overlapping relatively closed
subdomains $\GDir$ and $\GDiss$. For the sake of simplicity, we will assume that $\GDiss$ is
of class $C^{2,1}$, while  $\GDir$ is of class $C^{1,1}$, or the boundary of a polytope.
Finally, we denote by
\begin{equation}
\label{eq_radius_omega}
\RRR
\eq
\frac{1}{2} \sup_{\bx,\by \in \Omega} |\bx-\by|
\end{equation}
the radius of the domain $\Omega$.

The notation $\hOmega = (1/\RRR) \Omega$ and $\hGDiss \eq (1/\RRR) \GDiss$ will be also useful.
These denote rescaled version of $\Omega$ and $\GDiss$ and only depend on the ``shape'' of the
domain, not on its actual ``size''.

\subsection{Tensor notations}

If $\bv: \Omega \to \C^d$ is a smooth vector-valued function,
we define its gradient $\tgrad \bv: \Omega \to \C^{d \times d}$ by setting
\begin{equation*}
(\tgrad \bv)_{j\ell} \eq \partial_j \bv_\ell
\end{equation*}
for all $1 \leq j,\ell \leq d$, and we extend this definition in the sense
of distribution when appropriate. The symmetrized gradient (or strain tensor)
of $\bv$ is then defined as
\begin{equation*}
\eps(\bv) \eq \frac{1}{2} \left \{\tgrad \bv + (\tgrad \bv)^T\right \}.
\end{equation*}

If $\ttau,\txi \in \C^{d \times d}$ are two tensors, we employ the standard notations
\begin{equation}
\label{eq_frobenius}
\ttau:\txi \eq \sum_{j,\ell=1}^d \ttau_{j\ell}\txi_{j\ell}
\qquad
|\ttau|^2 \eq \ttau:\overline{\ttau}
\end{equation}
for the Frobenius inner-product and norm on $\C^{d \times d}$.

\subsection{Functional setting}

If $D \subset \R^d$ is an open set, we classically denote by $L^2(D)$
the Lebesgue space of square-integrable complex valued functions \cite{adams_fournier_2003a}.
We further respectively employ the notations $\BL^2(D)$ and $\tens{L}^2(D)$
for the space of vector-valued and tensor-valued functions that have all components
in $L^2(D)$. We respectively denote by $(\cdot,\cdot)_D$ and $\|{\cdot}\|_D$ the standard
inner-products and norms of $L^2(D)$, $\BL^2(D)$ and $\tens{L}^2(D)$. Classically, for the
inner-product of $\tens{L}^2(D)$, the Frobenius inner-product introduced
at \eqref{eq_frobenius} is employed point-wise. If $\mathfrak{m}: D \to [0,+\infty]$
is measurable function, we introduce the functionals
$\|{\cdot}\|_{\mathfrak{m},D} \eq \sqrt{(\mathfrak{m}\cdot,\cdot)}$. The
integral is always well-defined (although possibly infinite) due to the positivity of the
integrand. If $\mathfrak{m}$ is uniformly bounded away from $0$ and $+\infty$, then
$\|{\cdot}\|_{\mathfrak{m},D}$ is a norm equivalent to the usual one $\|{\cdot}\|_D$.

We also employ standard notations for Sobolev spaces \cite{adams_fournier_2003a},
namely
\begin{equation*}
H^1(D)
\eq
\left \{
v \in L^2(D) \; | \; \grad v \in \BL^2(\Omega)
\right \}
\end{equation*}
and
\begin{equation*}
H^2(D)
\eq
\left \{
v \in H^1(D) \; | \; \tgrad \grad v \in \tens{L}^2(\Omega)
\right \}.
\end{equation*}
Similarly, $\BH^1(D)$ and $\BH^2(D)$ respectively contains vector-valued functions
having in component in $H^1(D)$ and $H^2(D)$. If $\gamma \subset \partial D$
is a relatively open subset of the boundary of $D$, then
\begin{equation*}
H^1_\gamma(D)
\eq
\{
v \in H^1(D) \; | \; v|_{\gamma} = 0
\},
\end{equation*}
where the restriction operator is understood in the sense of traces, see e.g.
\cite[Section 5.34]{adams_fournier_2003a}. We also note $\BH^1_\gamma(D) \eq [H^1_\gamma(D)]^d$
for vector-valued functions.

If $\Gamma \subset \R^d$ is a Lipschitz $d-1$ manifold, we can define $L^2(\Gamma)$ and
$\BL^2(\Gamma)$ using the surface measure, with their usual norm and inner product denoted
by $\|{\cdot}\|_{\Gamma}$ and $(\cdot,\cdot)_\Gamma$. If
$\underline{\mathfrak{M}}: \Gamma \to \R^{d \times d}$
is an measurable SPD matrix-valued function uniformly bounded away from $0$ and $+\infty$,
then the application
$\|{\cdot}\|_{\underline{\mathfrak{M}},\Gamma} \eq \sqrt{(\underline{\mathfrak{M}}\cdot,\cdot)_\Gamma}$
is an equivalent norm on $\BL^2(\Gamma)$.

For $1 < s < 2$, we will also use the fractional Sobolev space
\begin{equation*}
H^s(\Omega)
\eq
\left \{
v \in H^1(\Omega)
\; \left |\;
\int_\Omega \int_\Omega
\frac{|\grad v(\bx)-\grad v(\by)|^2}{|\bx-\by|^{2s+d-2}}
d\by d\bx
<
+\infty
\right .
\right \}
\end{equation*}
and its vector-valued version $\BH^s(\Omega)$. Crucially, if
$3/2 < s < 2$, then the trace operator
$\bv \in \BH^s(\Omega) \to (\tgrad \bv)|_{\partial \Omega} \in \tens{L}^2(\partial \Omega)$
is continuous.

\subsection{Notations related to the boundary}

If $\bx \in \GDiss$, the following decomposition of $\bxi \in \C^d$ will be useful:
\begin{equation*}
\bxi_N(\bx) \eq (\bxi \cdot \bn(\bx)) \bn(\bx)
\qquad
\bxi_T(\bx) \eq \bxi-\bxi_N(\bx).
\end{equation*}
Similarly for $\varphi \in H^s(\Omega)$, with $s \in (3/2,2]$, we can define
\begin{eqnarray*}
\grad_T \varphi = \grad\varphi-(\grad \varphi \cdot \bn) \bn \in \BL^2(\partial \Omega).
\end{eqnarray*}

We further note that under the assumption that $\GDiss$ is $C^{2,1}$,
$\bn$ can be extended into a $C^{1,1}$ vector field in a neighbourhood
$V \subset \Omega$ of $\GDiss$, and therefore for $\bv \in \BH^s(\Omega)$,
for some $s \in (3/2, 2]$, $\bv \cdot \bn$ (resp. $\bv_T$) will belong to $H^{s}(\Omega\cap V)$ 
(resp. $\BH^{s}(\Omega\cap V)$. Consequently we can define
$\grad_T(\bv \cdot \bn)$ and $\grad_T \cdot \bv_T$ in $\BL^2(\GDiss)$.

\subsection{Material coefficients}

The elastic properties of $\Omega$ are described by three real-valued functions
$\rho,\mu,\lambda: \Omega \to \R$. The three functions are assumed to be measurable.
Physically, $\rho$ represent the density of the material contained in $\Omega$,
whereas $\mu$ and $\lambda$ are the Lam\'e parameters respectively quantifying the
response of the material to shear and pressure constraints.

The shear and pressure wave speeds are respecitively defined by
\begin{equation}
\label{eq_definition_velV}
\velsV=\sqrt{\frac{\mu}{\rho}}, 
\qquad
\velpV=\sqrt{\frac{\lambda+2\mu}{\rho}}.
\end{equation}

Throughout this document, if $\phi$ is a measurable function,
the notation $\phi_{\min}$ and $\phi_{\max}$ respectively denote the
essential minimum and maximum of $\phi$. Classically, we shall assume hereafter
that $\rho_{\min} > 0$ and $\rho_{\max} < +\infty$, and similarly, that
$\mu_{\min} > 0$ and $\mu_{\max} < +\infty$. For the
remaining parameter, we demand that $\lambda_{\min} \geq 0$
and $\lambda_{\max} < + \infty$.

\begin{remark}[Relaxed assumptions on $\lambda$]
We could, in principle, relax our assumptions on $\lambda$
by only requiring that $\lambda_{\max} \leq +\infty$ and
$(2\mu+d\lambda)_{\min} > 0$. The first relaxation is essentially
already taken care of in the manuscript, since most of the bounds
we will establish are uniform in $\lambda_{\max}$. For simplicity,
however, we do not consider the possibility of negative values for
$\lambda$ in this work.
\end{remark}

\subsection{Frequency}

In the remainder of this work, $\omega \geq 0$ denotes the angular frequency.
We can think of $\omega$ being fixed, but large. All our estimates will be explicit
in $\omega$. The notation
\begin{equation}
\label{eq_definition_ks}
\ks \eq \frac{\omega\RRR}{\vels}.
\end{equation}
where
\begin{equation}
\label{eq_definition_vels}
\vels \eq \sqrt{\frac{\mu_{\min}}{\rho_{\max}}},
\end{equation}
will be extremely useful. Notice $\ks$ is adimensional, it is proportional
to the ratio between the shear wavelength and the domain size, i.e. it measures
the size of the domain in number of shear wavelength.

\subsection{Dissipative boundary}

The dissipative boundary conditions are described with the help of
a symmetric matrix-valued function $\AAA \in \tens{C}^{0,1}(\overline{\GDiss})$.
The quantities
\begin{equation*}
a_{\min} \eq \operatorname{ess} \inf_{\bx \in \GDiss} \min_{\substack{\bxi \in \R^d \\ |\bxi| = 1}}
\AAA(\bx) \bxi \cdot \bxi,
\qquad
a_{\max} \eq \operatorname{ess} \sup_{\bx \in \GDiss} \max_{\substack{\bxi \in \R^d \\ |\bxi| = 1}}
\AAA(\bx) \bxi \cdot \bxi
\end{equation*}
respectively represent the minimal and maximal eigenvalues of $\AAA$, and we will
assume that $0 < a_{\min}$ and $a_{\max} < +\infty$. We will sometimes assume that
$\AAA$ satisfies
\begin{equation}
\label{eq_assumption_AAA}
\AAA(\bx)\bxi = a_T \bxi_T + a_N \bxi_N
\end{equation}
for two constant real numbers $a_T,a_N$, whereby $a_{\min} = \min(a_T,a_N)$
and $a_{\max} = \max(a_T,a_N)$.

The ``adimensional'' matrix $\al \eq (\rho_{\max}\mu_{\min})^{-1/2}\AAA$
will often be employed. The notations
$\alpha_{\min} \eq (\rho_{\max}\mu_{\min})^{1/2} a_{\min}$
and
$\alpha_{\max} \eq (\rho_{\max}\mu_{\min})^{1/2} a_{\max}$
are used for the smallest and largest eigenvalues of $\al$,
and we set $\alpha_T \eq (\rho_{\max}\mu_{\min})^{-1/2} a_T$ and
$\alpha_N \eq (\rho_{\max}\mu_{\min})^{-1/2} a_N$ whenever relevant.

\subsection{Stress tensor}

For $\bv \in \BH^1(\Omega)$, the stress tensor is defined as
\begin{equation*}
\sig(\bv) \eq 2\mu\eps(\bv) + \lambda(\div\bv)\tens{I}.
\end{equation*}
For $\bv,\bw \in \BH^1(\Omega)$, we will often use the facts that
\label{eq_stress_identities}
\begin{equation}
\sig(\bv):\grad \overline{\bw} = \sig(\bv):\eps(\overline{\bw}),
\qquad
\sig(\bv):\eps(\overline{\bv}) = 2\mu|\eps(\bv)|^2 + \lambda|\div \bv|^2
\end{equation}
almost everywhere in $\Omega$.

\subsection{Model problem}

Classically, we recat the model problem given in \eqref{eq_helmholtz_strong}
into the following weak form: Given $\fff \in \BL^2(\Omega)$,
find $\uuu \in \BH^1_{\GDir}(\Omega)$ such that
\begin{equation}
\label{eq_helmholtz_weak}
-
\omega^2 (\rho\uuu,\bv)_\Omega
-
i\omega(\AAA\uuu,\bv)_{\GDiss}
+
2(\mu \eps(\uuu),\eps(\bv))_\Omega
+
(\lambda \div \uuu,\div \bv)_\Omega
=
(\rho\fff,\bv)_\Omega
\end{equation}
for all $\bv \in \BH^1_{\GDir}(\Omega)$.

%%%   We record for future use that
%%%   \begin{equation}
%%%   \label{eq_a_mu_vels}
%%%   (a_T \mu_{\min})^{1/2} = \alpha_T \vels^{1/2},
%%%   \qquad
%%%   (a_N \mu_{\min})^{1/2} = \alpha_N \vels^{1/2}
%%%   \end{equation}
%%%   and
%%%   \begin{equation}
%%%   \label{eq_rho_vels}
%%%   \rho_{\max}
%%%   \leq
%%%   \frac{1}{\alpha_{\min} \vels} a_{\min}.
%%%   \end{equation}

\section{Standard key identities}
\label{section_standard}

This section gathers key integration by parts techniques that will be useful
throughout the manuscript. Throughout this section, $\fff \in \BL^2(\Omega)$ is fixed,
and $\bh \in \BC^1(\overline{\Omega}, \R)$ is a smooth real-valued vector field.
We also assume in this section that $\rho,\mu,\lambda \in C^{0,1}(\overline{\Omega})$,
and introduce the notation
\begin{equation*}
\LV_{\bh}(\phi) \eq \frac{\bh \cdot \grad \phi}{\phi}
\end{equation*}
for non-negative functions $\phi \in C^{0,1}(\overline{\Omega})$,
with the convention that $\LV_{\bh}(\phi) = 0$ whenever $\phi = 0$.

\subsection{G\aa rding identites}

We start by recording two simple identities that are classically obtained by
selecting the test function $\bv = \uuu$ in \eqref{eq_helmholtz_weak} and by taking
the real and imaginary parts of the resulting identity.

\begin{lemma}[G\aa rding-type identities]
For all $\uuu \in \BH^1_{\GDir}(\Omega)$ solution to \eqref{eq_helmholtz_weak},
we have
\begin{subequations}
\begin{equation}
\label{eq_garding_real}
2\|\eps(\uuu)\|_{\mu,\Omega}^2
+
\|\div \uuu\|_{\lambda,\Omega}^2
=
\Re (\rho\fff,\uuu)_\Omega
+
\omega^2 \|\uuu\|_{\rho,\Omega}^2
\end{equation}
and
\begin{equation}
\label{eq_garding_imag}
\omega \|\uuu\|_{\AAA,\GDiss}^2
=
\Im (\rho\fff,\uuu)_\Omega.
\end{equation}
\end{subequations}
\end{lemma}

\subsection{Integration by parts identities}

We next establish a Rellich-type identity. For constant Lam\'e parameters $\mu$ and $\lambda$,
such an identity is established in \cite[Lemma 4]{cummings_feng_2006a}. It is extended here to the
variable coefficients case. We skip the proof here for the sake of shortness, as the
modifications essentially follow the lines of \cite[Lemma 4.2]{graham_pembery_spence_2019a}.

\begin{lemma}[Rellich identity]
For all $\bv \in \BH^2(\Omega)$, we have
\begin{equation}
\label{eq_rellich_identity}
-2\Re (\div \sig(\bv),(\bh\cdot\grad)\bv)_\Omega
=
-
\LR_{\bh,\Omega}(\bv)
-
\LR_{\bh,\GDiss}(\bv)
+
\LB_{\bh,\GDir}(\bv)
+
\LB_{\bh,\GDiss}(\bv)
\end{equation}
with
\begin{align*}
\LR_{\bh,\Omega}(\bv)
&\eq
\int_\Omega \left \{
(\div \bh + \LV_{\bh}(\mu))2\mu|\eps(\bv)|^2
+
(\div \bh + \LV_{\bh}(\lambda))\lambda|\div \bv|^2
\right \}
\\
&-2\Re\int_\Omega \left \{
2
\mu \eps(\bv) : (\tgrad \bh \tgrad \overline{\bv})
+
\lambda (\div \bv) \tgrad \bh : (\tgrad \overline{\bv})^T
\right \},
\\
\LB_{\bh,\GDir}(\bv)
&\eq
\int_{\GDir} (\bh \cdot \bn) \sig(\bv):\eps(\overline{\bv})
-2 \Re
\int_{\GDir}
(\sig(\bv)\bn) \cdot ((\bh \cdot \grad) \overline{\bv}),
\end{align*}
and
\begin{equation*}
\LB_{\bh,\GDiss}(\bv)
\eq
\int_{\GDiss} (\bh \cdot \bn) \sig(\bv):\eps(\overline{\bv}),
\quad
\LR_{\bh,\GDiss}(\bv)
\eq
2 \Re \int_{\GDiss}
(\sig(\bv)\bn) \cdot ((\bh \cdot \grad) \overline{\bv}).
\end{equation*}
\end{lemma}

\begin{corollary}[Rellich identity with lower regularity]
The Rellich identity in \eqref{eq_rellich_identity} remains valid for $\bv \in \BH^s(\Omega)$,
for some $s\in (3/2, 2]$, where the first term of the right-hand side of
\eqref{eq_rellich_identity} has to be understood as a duality bracket.
\end{corollary}

\begin{proof}
We first notice that the absolute value of each term of the identity \eqref{eq_rellich_identity}
is bounded by $\|\bv\|_{\BH^s(\Omega)}^2$ up to a multiplicative factor independent of $\bv$.
The only subtle term is $(\div \sig(\bv),(\bh\cdot\grad)\bv)_\Omega$, since 
in that case $\div \sig(\bv)\in \BH^{s-2}(\Omega)$ and need not be square integrable.
However, since $\bh \in \BC^1(\overline{\Omega})$, we have
$(\bh\cdot\grad)\bv\in \BH^{s-1}(\Omega)$, and
since $H^{s-1}(\Omega)$ is continuously embedded into $H^{2-s}(\Omega)=H^{2-s}_0(\Omega)$,
the term $(\div \sig(\bv),(\bh\cdot\grad)\bv)_\Omega$ can be replaced by
\[
\langle \div \sig(\bv); (\bh\cdot\grad)\bv\rangle_{\BH^{s-2}(\Omega)-\BH^{2-s}_0(\Omega)},
\]
recalling that $\BH^{2-s}_0(\Omega)'=\BH^{s-2}(\Omega)$. We then obtain
\eqref{eq_rellich_identity} by a density arguments since $H^s(\Omega)$ is dense
in $H^2(\Omega)$.
\end{proof}

The Rellich identity in \eqref{eq_rellich_identity} will be useful to deal
with the second-order term in \eqref{eq_helmholtz_weak}. For the zero-order
term, we need another identity that extends \cite[Lemma 5]{cummings_feng_2006a}
to the case of a variable density $\rho$. The proof proceed
component-wise as in acoustic case, so that it is similar
to, e.g., \cite[Proposition 1]{barucq_chaumontfrelet_gout_2017a},
\cite[Theorem 1]{brown_gallistl_peterseim_2017a}, \cite[Theorem B.2]{chaumontfrelet_nicaise_2018a}
or \cite[Lemma 4.2]{graham_pembery_spence_2019a}. Hence, we de not reproduce the
proof here for the sake of shortness.

\begin{lemma}[Zero-order term identity]
For all $\bv \in \BH^1(\Omega)$ and $\bh \in \BC^1(\overline{\Omega})$, we have
\begin{eqnarray}
\label{eq_identity_mass}
-2\Re \int_\Omega
\rho
\bv \cdot (\bh\cdot \grad) \overline{\bv}
=
\int_\Omega (\div \bh + \LV_{\bh}(\rho)) \rho|\bv|^2
-
\int_{\partial \Omega} (\bh \cdot \bn) \rho |\bv|^2.
\end{eqnarray}
\end{lemma}

%%%  \begin{proof}
%%%  \begin{align*}
%%%  2\Re \int_\Omega
%%%  \rho
%%%  \bv \cdot (\bh \cdot \grad \overline{\bv}) 
%%%  &=
%%%  2 \Re
%%%  \sum_{j=1}^d
%%%  \int_\Omega
%%%  \rho \bv_j (\bh \cdot \grad \overline{\bv})_j
%%%  \\
%%%  &=
%%%  2 \Re
%%%  \sum_{j,\ell=1}^d
%%%  \int_\Omega
%%%  \rho \bv_j \bh_\ell \partial_\ell \overline{\bv}_j
%%%  \\
%%%  &=
%%%  \sum_{j,\ell=1}^d
%%%  \int_\Omega
%%%  \rho \bh_\ell \partial_\ell |\bv_j|^2
%%%  \\
%%%  &=
%%%  \sum_{j,\ell=1}^d
%%%  \left \{
%%%  \int_{\partial \Omega} \rho \bh_\ell |\bv_j|^2 \bn_\ell
%%%  -
%%%  \int_\Omega
%%%  \partial_\ell(\rho \bh_\ell) |\bv_j|^2
%%%  \right \}
%%%  \\
%%%  &=
%%%  \int_{\partial \Omega} \rho |\bv|^2 \bh \cdot \bn
%%%  -
%%%  \int_\Omega \div (\rho \bh) |\bv|^2
%%%  \\
%%%  &=
%%%  \int_{\partial \Omega} \rho |\bv|^2 \bh \cdot \bn
%%%  -
%%%  \int_\Omega (\rho \div \bh + \bh \cdot \grad \rho) |\bv|^2
%%%  \end{align*}
%%%  \end{proof}

\subsection{Dirichlet boundary}

We then need to take care of the boundary term in \eqref{eq_rellich_identity}.
To manage the Dirichlet boundary condition on $\GDir$, we can simply use
\cite[Identity (19)]{bey_heminna_loheac_2003a}, which we record below. The proofs in
\cite{bey_heminna_loheac_2003a} is carried out for the case of constant Lam\'e parameters $\mu$
and $\lambda$, but a careful inspection of the proof shows that it holds verbatim
for the case of variable coefficients (no integration by parts is performed).

\begin{lemma}[Dirichlet boundary]
For all $\bv \in \BH^1_{\GDir} \cap \BH^s(\Omega)$ with $s \in (3/2,2]$, we have
\begin{equation*}
\LB_{\bh,\GDir}(\bv)
=
-\int_{\GDir} (\bh \cdot \nn) \left \{
(\mu |\partial_n \bv_T|^2 + (2\mu+\lambda)  |\partial_n \bv_n|^2
\right \}.
\end{equation*}
\end{lemma}

Since $\LB_{\bh,\GDir}(\bv) \geq 0$ whenever $\bh \cdot \bn \leq 0$ and $\bv = \bo$ on $\GDir$,
this term will typically provides some additional control on the solution in favorable geometries.

\subsection{Key identity}

We can now provide the key identity central to the derivation of stability estimates.

\begin{theorem}[Morawetz identity]
\label{theorem_morawetz_identity}
Assume that  $\uuu \in \BH^1_{\GDir}(\Omega)$ is a solution to \eqref{eq_helmholtz_weak}.
Then $\uuu \in \BH^s(\Omega)$ for some $s \in (3/2,2]$ and the following identity holds true:
\begin{multline}
\label{eq_morawetz_identity}
\omega^2 \int_\Omega (\div \bh + \LV_{\bh}(\rho))\rho|\uuu|^2
+
\LB_{\bh,\GDiss}(\uuu)
+
\LB_{\bh,\GDir}(\uuu)
=
\\
2\Re (\rho\fff,(\bh\cdot\grad)\uuu)_\Omega
+
\omega^2 \int_{\GDiss} (\bh\cdot\bn)\rho|\uuu|
+
\LR_{\bh,\GDiss}(\uuu) + \LR_{\bh,\Omega}(\uuu).
\end{multline}
\end{theorem}

\begin{proof}
Assume that $\uuu \in \BH^1_{\GDir}(\Omega)$ solves \eqref{eq_helmholtz_weak}.
Then, we can rewrite \eqref{eq_helmholtz_weak} as follows:
\begin{equation}
\label{eq_helmholtz_weak_rhsdiff}
2(\mu \eps(\uuu),\eps(\bv))_\Omega
+
(\lambda \div \uuu,\div \bv)_\Omega
=
(\rho\fff,\bv)_\Omega
+
\omega^2 (\rho\uuu,\bv)_\Omega
+
i\omega(\AAA\uuu,\bv)_{\GDiss}
\end{equation}
for all $\bv \in \BH^1_{\GDir}(\Omega)$. In other words, $\uuu$ is solution
to the elastostatic problem with load term $\rho \fff+\rho\omega^2\uuu \in \BL^2(\Omega)$
and traction term $i\omega\AAA\uuu\in \BH^{1/2}(\GDiss)$, since
$\AAA \in \tens{C}^{0,1}(\overline{\GDiss})$.
As a result, under the assumption on $\GDir$ and $\GDiss$ stated in Section \ref{section_domain}
and the assumption that $\mu,\lambda \in C^{1,1}(\overline{\Omega})$
and $\rho\in L^\infty(\Omega)$, standard regularity results for the elasticity system ensure
that $\uuu \in \BH^2(\Omega)$ if $\GDir$ is $C^{1,1}$.

Alternatively if $\GDir$ is the boundary of a polytope, again by standard regularity
results for the elasticity system $\uuu \in \BH^2(\Omega\setminus \mathcal{O})$,
for any open neighbourhood $\CO$ of $\GDir$, while 
$\uuu \in \BH^s(\Omega\cap \CO)$
with $s \in (3/2,2]$, due to
Theorem 17.13 of \cite{dauge_1988a} 
since there is no corner singularity exponent $\alpha$ in the strip $\Re\alpha\in [0,1/2]$
due to \cite[\S 5]{dauge_1989a} if $d=2$, while there is no edge singularity exponent $\alpha$
in the strip $\Re\alpha\in [0,1/2]$ due to \cite[\S 5]{dauge_1989a} and no corner singularity
exponent $\alpha$ in the strip $\Re\alpha\in [-1,0]$ due to \cite[Theorem 3]{kozlov_mazya_1988a}
if $d=3$ (see also~\cite{nicaise_1992}). Altogether this yields $\uuu \in \BH^s(\Omega)$
with $s \in (3/2,2]$ (since $\CO$ is an arbitrary open neighbourhood  of $\GDir$).

Once the regularity $\uuu \in \BH^s(\Omega)$ for some  $s \in (3/2,2]$ is obtained,
the identity stated in \eqref{eq_helmholtz_weak_rhsdiff} follows from the addition
of the Rellich identity in \eqref{eq_rellich_identity} with the multiplication of
\eqref{eq_identity_mass} by $\omega^2$ (both applied to $\uuu$).
\end{proof}

\subsection{Key assumptions on the multiplier}
\label{section_key_assumptions}

For arbitrary fields $\bh$, the identity stated in \eqref{eq_morawetz_identity}
is not useful, as the key terms of the identity may not have definite signs.
Here, we provide a first set of assumptions on $\bh$ that enables useful applications
of \eqref{eq_morawetz_identity}.

The prototypical scenario where \eqref{eq_morawetz_identity} provides an interesting
estimate is when $\GDir = \partial D$ for a closed set $D \subset B_\RRR$ star-shaped
with respect to $\bzero$, $\GDiss = \partial B_\RRR$, $\Omega = B_\RRR \setminus D$,
and the coefficients $\rho$, $\mu$ and $\lambda$ are constant. In this case, selecting
$\bh(\bx) = \bx$, it is clear that $\|\bh\|_{\BL^\infty(\GDiss)} = \bh \cdot \bn = \ell$
on $\GDiss$, and we easily see that
\begin{equation*}
d\|\bv\|_{\rho,\Omega}^2
=
\int_\Omega(\div \bh + \LV_{\bh}(\rho)) \rho|\bv|^2
\end{equation*}
and
\begin{equation*}
\LR_{\bh,\Omega}(\bv) = (d-2) \left \{
2\|\eps(\bv)|_{\mu,\Omega}^2 + \|\div \bv\|_{\lambda,\Omega}^2
\right \}.
\end{equation*}
Besides, we have $\bh \cdot \bn \leq 0$ on $\GDir$ since $D$ is star-shaped with
respect to $\bzero$.

\begin{subequations}
\label{eq_assumptions_multiplier}
These observations motivate our key assumptions on $\bh$, which
slightly relax the identities derived above. We assume that there
exists constants $\eta,\varepsilon \geq 0$ with
\begin{equation}
\label{eq_assumption_gamma}
2\gamma \eq 2 - \eta - \varepsilon > 0,
\end{equation}
such that the estimates
\begin{equation}
\label{eq_assumption_L2_Omega}
(d-\eta)\|\bv\|_{\rho,\Omega}^2
\leq
\int_\Omega(\div \bh + \LV_{\bh}(\rho)) \rho|\bv|^2
\end{equation}
and
\begin{equation}
\label{eq_assumption_LR_Omega}
|\LR_{\bh,\Omega}(\bv)|
\leq
(d-2+\varepsilon) \left \{
2\|\eps(\bv)\|_{\mu,\Omega}^2 + \|\div \bv\|_{\lambda,\Omega}^2
\right \}.
\end{equation}
hold true. We also require that
\end{subequations}
\begin{equation}
\label{eq_assumption_sign_boundary}
\bh \cdot \bn \leq 0 \text{ on } \GDir,
\qquad
\bh \cdot \bn \geq 0 \text{ on } \GDiss.
\end{equation}

For the sake of simplicity, we introduce the notation
\begin{subequations}
\label{eq_bounds_h}
\begin{equation}
\label{eq_upper_bounds_h}
\MMM \eq \max(1,\RRR^{-1} \|\bh\|_{\BL^\infty(\Omega)})
\qquad
\nu
\eq
\frac{1}{\MMM} \|\tgrad \bh\|_{\tens{L}^\infty(\Omega)}
\end{equation}
and
\begin{equation}
\label{eq_lower_bounds_h}
\mmm
\eq
\frac{1}{\ell} \min_{\bx \in \GDiss} \bh(\bx) \cdot \bn (\bx).
\end{equation}
\end{subequations}
Notice that $M = 1$ if $\bx = \bh$, and
that $m = 1$ if we further assume that $\GDiss$ is a sphere of radius $\RRR$.

\section{New estimates}
\label{section_new}

This section now introduces a set of original estimates to efficiently deal
with the Robin condition on $\GDiss$, and avoid the use of the Korn constant
of the domain.

\subsection{Second fundamental form}

We introduce the second fundamental form $\tens{B} \in \tens{L}^\infty(\GDiss)$ of $\GDiss$,
namely
\begin{equation*}
\tens{B} \eq \tgrad_T \bn + (\grad_T \cdot \bn) \bn \otimes \bn,
\end{equation*}
as well as the constants
\begin{equation*}
\beta_T = \RRR \max_{\GDiss}
\sup_{\substack{\bxi \in \R^d \\ |\bxi| = 1}} |(\tgrad_T \bn) \bxi \cdot \bxi|,
\qquad
\beta_N = \RRR \max_{\GDiss} (\grad_T \cdot \bn),
\end{equation*}
where the factor $\RRR$ is chosen to make $\beta_T$ and $\beta_N$
adimensional. For later references, we observe that
\begin{equation}
\label{eq_bound_B}
|(\tens{B}\bv,\bv)_{\GDiss}|
\leq
\frac{1}{\ell} \left \{
\beta_T \|\bv_T\|_{\GDiss}^2 + \beta_N \|\bv \cdot \bn\|_{\GDiss}^2
\right \}.
\end{equation}
We also introduce the quantity
\begin{equation}
\label{eq_definition_chi}
\chi
\eq
\max \left \{
\frac{\beta_T}{\alpha_T},
\frac{\beta_N}{\alpha_N}
\right \}
\end{equation}
which measures the ratio between and $\AAA$ and $\BBB$ when $\AAA$ is of the form
\eqref{eq_assumption_AAA}, since it naturally appears in the estimates below.

\subsection{Korn inequality}

In this section, we establish a Korn inequality for functions
in $\BH^s(\Omega)$ for some  $s \in (3/2,2]$. Notice that because
the Dirichlet boundary conditions are only imposed on a subset of
$\partial \Omega$, naively using the Korn  inequality generally
results in an unknown constant. Here, we will use the additional information
of the solution on $\GDiss$ to obtain a fully-explicit inequality.

\begin{proposition}[Korn inequality]
For all $\bv \in \BH^1_{\GDir}(\Omega) \cap \BH^s(\Omega)$ with $s \in (3/2,2]$, we have
\begin{equation}
\label{eq_basic_korn_inequality}
\|\grad \bv\|_\Omega^2
\leq
2\|\eps(\bv)\|_\Omega^2
+
\frac{1}{\RRR}\left (
\beta_T\|\bv_T\|_{\GDiss}^2 + \beta_N\|\bv\cdot\bn\|_{\GDiss}^2
\right )
+
2\|\bv \cdot \bn\|_{\GDiss}\|\grad_T \cdot \bv_T\|_{\GDiss}.
\end{equation}
\end{proposition}

\begin{proof}
We start with \cite[Theorem 3.1.1.1]{grisvard_1985a}
stating that
\begin{equation*}
\|\div \bv\|_\Omega
-
(\grad \bv,(\grad \bv)^T)
=
-2 (\bv_T, \grad_T(\bv\cdot\bn))_{\GDiss} - (\BBB\bv,\bv)_{\GDiss}.
\end{equation*}
Note that  \cite[Theorem 3.1.1.1]{grisvard_1985a} is valid if $\GDiss$ is only $C^{1,1}$,
see \cite[p. 133]{grisvard_1985a}, while the contribution to $\GDir$ is zero due to the
Dirichlet boundary conditions on  $\GDir$. We then integrate by parts on $\GDiss$ the first
term in the right-hand side, leading to
\begin{equation*}
-
(\grad \bv,(\grad \bv)^T)
=
-\|\div \bv\|_\Omega^2 + 2 (\grad_T \cdot \bv_T,\bv\cdot\bn)_{\GDiss} - (\BBB\bv,\bv)_{\GDiss},
\end{equation*}
hence
\begin{equation*}
-\Re (\grad \bv,(\grad \bv)^T)
\leq
2 \|\grad_T \cdot \bv_T\|_{\GDiss}\|\bv\cdot\bn\|_{\GDiss}
+
\frac{1}{\RRR} \left (
\beta_T \|\bv_T\|_{\GDiss}^2
+
\beta_N\|\bv\cdot\bn\|_{\GDiss}^2
\right ),
\end{equation*}
where we used \eqref{eq_bound_B}. Then, \eqref{eq_basic_korn_inequality}
follows after observing that
\begin{equation*}
\|\grad \bv\|_\Omega^2
=
2\|\eps(\bv)\|_\Omega^2
-
(\grad \bv,(\grad \bv)^T)_\Omega.
\end{equation*}
\end{proof}

We also derive a coefficient-weighted version of Korn inequality \eqref{eq_basic_korn_inequality}.

\begin{corollary}[Weighted Korn inequality]
\label{corollary_korn_inequality}
Assume that $\AAA$ satisfies \eqref{eq_assumption_AAA}. Then the inequality
\begin{multline}
\label{eq_weighted_korn_inequality}
\mu_{\min} \|\grad \bv\|_\Omega^2
\leq
2\|\eps(\bv)\|_{\mu,\Omega}^2
+
\chi \ks^{-1} \omega\|\bv\|_{\AAA,\Omega}^2
\\
+
2\alpha_N^{-1/2}\ks^{-1/2}\omega^{1/2}\|\bv\|_{\AAA,\GDiss}\RRR^{1/2}\|\eps(\bv)\|_{\mu,\GDiss}.
\end{multline}
holds true for all $\bv \in \BH^1_{\GDir}(\Omega) \cap \BH^s(\Omega)$ with $s \in (3/2,2]$.
\end{corollary}

\begin{proof}
Let $\bv \in \BH^s(\Omega)$ for some  $s \in (3/2,2]$.
We start by multiplying \eqref{eq_basic_korn_inequality}
by $\mu_{\min}$, leading to
\begin{multline*}
\mu_{\min} \|\grad \bv\|_\Omega^2
\leq
2\|\eps(\bv)\|_{\mu,\Omega}^2
+
\frac{\mu_{\min}}{\RRR}\left (
\beta_T\|\bv_T\|_{\GDiss}^2 + \beta_N\|\bv\cdot\bn\|_{\GDiss}^2
\right )
\\
+
2\mu_{\min}^{1/2} \|\bv \cdot \bn\|_{\GDiss}\|\grad_T \cdot \bv_T\|_{\mu,\GDiss}.
\end{multline*}
We then observe that,
\begin{equation*}
\mu_{\min} \beta_T
=
\mu_{\min} \frac{\beta_T}{\sqrt{\mu_{\min}\rho_{\max}}\alpha_T} a_T
=
\vels \frac{\beta_T}{\alpha_T} a_T,
\end{equation*}
and that similarly
\begin{equation*}
\mu_{\min} \beta_N
=
\vels \frac{\beta_N}{\alpha_N} a_N.
\end{equation*}
Recalling assumption \eqref{eq_assumption_AAA} on $\AAA$,
this leads to
\begin{align*}
\frac{\mu_{\min}}{\RRR}\left (
\beta_T\|\bv_T\|_{\GDiss}^2 + \beta_N\|\bv\cdot\bn\|_{\GDiss}^2
\right )
&\leq
\frac{\vels}{\RRR}
\max \left \{\frac{\beta_T}{\alpha_T},\frac{\beta_N}{\alpha_N}\right \}
\|\bv\|_{\AAA,\GDiss}^2
\\
&=
\ks^{-1}
\max \left \{\frac{\beta_T}{\alpha_T},\frac{\beta_N}{\alpha_N}\right \}
\omega \|\bv\|_{\AAA,\GDiss}^2
\\
&=
\ks^{-1} \chi
\omega \|\bv\|_{\AAA,\GDiss}^2,
\end{align*}
recalling the definition of $\chi$ in \eqref{eq_definition_chi}.

To finish the proof, we use again \eqref{eq_assumption_AAA}
\begin{align*}
2\mu_{\min}^{1/2} \|\bv \cdot \bn\|_{\GDiss}\|\grad_T \cdot \bv_T\|_{\mu,\GDiss}
&\leq
2\sqrt{\frac{\mu_{\min}}{a_N}} \|\bv\|_{\AAA,\GDiss}\|\eps(\bv)\|_{\mu,\GDiss}
\\
&=
2 \alpha_N^{-1/2}\ks^{-1/2} \omega^{1/2}\|\bv\|_{\AAA,\GDiss}\RRR^{1/2}\|\eps(\bv)\|_{\mu,\GDiss}.
\end{align*}
\end{proof}

\subsection{New estimate for the Robin boundary}

This section presents a new estimate for $\LR_{\bh,\GDiss}(\bv)$
to be used as an alternative to the treatment proposed
in \cite{cummings_feng_2006a}. The strategy employed in \cite{cummings_feng_2006a}
leads to suboptimal results, and is detailed in Section \ref{section_general_robin}
below. Our new estimate requires that the impedance matrix $\AAA$ satisfies
\eqref{eq_assumption_AAA}, but in returns, provides an optimal
scaling in frequency for the stability constant.

\begin{proposition}[Identity on the Robin boundary]
\label{proposition_identity_robin_boundary}
Assume that $\bh \times \bn = \bzero$ on $\GDiss$ and that
$\AAA$ satisfies \eqref{eq_assumption_AAA}. Then for all 
$\bv \in \BH^1_{\GDir}(\Omega) \cap \BH^s(\Omega)$ with $s \in (3/2,2]$,
we have
\begin{multline}
\label{eq_identity_robin_boundary}
(\AAA\bv,(\bh\cdot\grad)\bv)_{\GDiss}
=
\\
2(\AAA\bv,\bh\eps(\bv))_{\GDiss}
+
(\AAA\bv,(\tgrad \bh)\bv)_{\GDiss}
+
a_T((\bh\cdot\bn)(\grad_T \cdot \bv_T),\bv\cdot\bn)_{\GDiss}
\\
-
a_N ((\bn\cdot\grad)\bh (\bv \cdot \bn),\bv)_{\GDiss}
-
a_N((\bh\cdot\bn)(\bv\cdot\bn),\eps(\bv)\bn \cdot \bn)_{\GDiss}.
\end{multline}
\label{prop_identityRobinbc}
\end{proposition}

\begin{proof}
Let us first notice that the assumption that $\bh \times \bn = \bzero$ on $\GDiss$
actually implies that
\begin{equation}
\label{tmp_properties_h}
\bh = (\bh \cdot \bn)\bn,
\end{equation}
on $\GDiss$. Then, the definition of $\eps$ shows that
\begin{equation*}
(\bh\cdot\grad)\bv = \bh\grad\bv = 2\bh\eps(\bv)-\bh(\grad\bv)^T.
\end{equation*}
To simplify the last term in the right-hand side, we observe that the identity
\begin{equation*}
\sum_{k=1}^d \bh_k \partial_j \bv_k
=
\sum_{k=1}^d \partial_j(\bh_k\bv_k)
-
\sum_{k=1}^d \partial_j \bh_k \bv_k
\end{equation*}
holds true for $1 \leq j \leq d$, so that we can write in compact form that
\begin{equation*}
-\bh(\grad\bv)^T
=
-\grad(\bh \cdot \bv) + (\grad \bh)\bv.
\end{equation*}
We next work on the first term in the right-hand side. Specifically, we have
\begin{align*}
-\grad(\bh \cdot \bv)
&=
-(\bn\cdot\grad)(\bh \cdot \bv) \bn - \grad_T (\bh\cdot\bv).
\\
&=
- ((\bn \cdot \grad) \bh \cdot \bv) \bn
-(\bh \cdot (\bn\cdot\grad) \bv) \bn
- \grad_T (\bh\cdot\bv).
\\
&=
-((\bn\cdot\grad) \bh \cdot \bv) \bn
-\bh \cdot \bn (\eps(\bv) \bn \cdot \bn)\bn
- \grad_T (\bh\cdot\bv)
\end{align*}
after using \eqref{tmp_properties_h} in the last equality, we indeed have
\[
(\bh \cdot (\bn\cdot\grad) \bv) \bn
=
(\bh \cdot \bn) (\bn \cdot (\bn \cdot \grad \bv))
\bn,
\]
and since
\[
\bn \cdot (\bn\cdot \grad \bv) = \eps(\bv) \bn \cdot \bn,
\]
the last identity holds.
At that point, we have established that  
\begin{align*}
-(\AAA\bv,\grad(\bh\cdot\bn))_{\GDiss}
=
&-
a_N (\bv\cdot\bn,(\bn\cdot\grad)\bh \cdot \bv)_{\GDiss}
\\
&-
a_N (\bv \cdot \bn,(\bh\cdot\bn) \eps(\bv)\bn \cdot \bn)_{\GDiss}
-
a_T (\bv_T,\grad_T(\bh\cdot\bv))_{\GDiss},
\end{align*}
and \eqref{eq_identity_robin_boundary} follows since
\begin{equation*}
-(\bv_T,\grad_T(\bh\cdot\bv))_{\GDiss}
=
(\grad_T \cdot \bv_T,\bh\cdot\bv)_{\GDiss}
=
((\bh\cdot\bn)\grad_T\cdot\bv_T,\bv\cdot\bn)_{\GDiss}
\end{equation*}
where we again used \eqref{tmp_properties_h}.
\end{proof}

\begin{lemma}[Robin boundary estimate]
\label{lemma_robin_boundary}
Assume that $\bh \times \bn = \bzero$ on $\GDiss$ and that
$\AAA$ satisfies \eqref{eq_assumption_AAA}. Then, the estimate
\begin{equation}
\label{eq_Robin0}
\omega \Im (\AAA\bv,(\bh\cdot\grad)\bv)
\leq
M \left \{
\Crob \ks^{1/2}\omega^{1/2}\|\bv\|_{\AAA,\GDiss}\ell^{1/2}\|\eps(\bv)\|_{\mu,\GDiss} + \zeta \omega \|\bv\|_{\AAA,\GDiss}^2
\right \}
\end{equation}
holds true for all $\bv \in \BH^1_{\GDir}(\Omega) \cap \BH^s(\Omega)$ with $s \in (3/2,2]$, and
\begin{equation*}
\Crob \eq
\left (
2 + \sqrt{\frac{\alpha_{\max}}{\alpha_{\min}}}
\right )
\sqrt{\alpha_{\max}},
\qquad
\zeta \eq 2\nu\sqrt{\frac{\alpha_{\max}}{\alpha_{\min}}}.
\end{equation*}
In particular, if $\uuu$ solves \eqref{eq_helmholtz_weak}, we have
\begin{equation}
\label{eq_estimate_LR_GDiss}
|\LR_{\bh,\GDiss}^{2}(\uuu)|
\leq
M \left \{
\Crob \ks^{1/2}\omega^{1/2}\|\uuu\|_{\AAA,\GDiss}\ell^{1/2}\|\eps(\bv)\|_{\mu,\GDiss}
+
\zeta \omega \|\uuu\|_{\AAA,\GDiss}^2
\right \}.
\end{equation}
\end{lemma}

\begin{proof}
We start by invoking \eqref{eq_identity_robin_boundary} and \eqref{eq_upper_bounds_h}
to show that
\begin{multline}
\frac{1}{M} |(\AAA\bv,(\bh\cdot\grad)\bv)_{\GDiss}|
\leq
2\RRR\|\AAA\bv\|_{\GDiss}\|\eps(\bv)\|_{\GDiss}
+
\nu \|\AAA\bv\|_{\GDiss}\|\bv\|_{\GDiss}
\\
+a_T\ell \|\grad_T \cdot \bv_T\|_{\GDiss}\|\bv\cdot\bn\|_{\GDiss}
+
a_N\nu \|\bv\cdot\bn\|_{\GDiss}\|\bv\|_{\GDiss}
\\
+
a_N\ell\|\bv\cdot\bn\|_{\GDiss}\|\eps(\bv)\bn \cdot \bn\|_{\GDiss}
\label{eq_Robin1}
\end{multline}

We then write that
\begin{equation*}
\|\AAA\bv\|_{\GDiss}\|\bv\|_{\GDiss}
\leq
\sqrt{\frac{\alpha_{\max}}{\alpha_{\min}}}\|\bv\|_{\AAA,\GDiss}^2,
\end{equation*}
and
\begin{equation*}
a_N \|\bv\cdot\bn\|_{\GDiss}\|\bv\|_{\GDiss}
\leq
\sqrt{\frac{\alpha_N}{\alpha_{\min}}}\|\bv\|_{\AAA,\GDiss}^2,
\end{equation*}
so that
\begin{equation}
\label{tmp_control_2nd_4th}
\Big (
\nu\|\AAA\bv\|_{\GDiss} + a_N\nu\|\bv\cdot\bn\|_{\GDiss}
\Big )
\|\bv\|_{\GDiss}
\leq
2\nu \sqrt{\frac{\alpha_{\max}}{\alpha_{\min}}}\|\bv\|_{\AAA,\GDiss}^2.
\end{equation}
After multiplying by $M\omega$, \eqref{tmp_control_2nd_4th} controls the second
and forth terms in the right-hand side of \eqref{eq_Robin1} by the last term in
the right-hand side of \eqref{eq_Robin0}. Thus, it remains to bound the remaining
terms by the first term in the right-hand side of \eqref{eq_Robin0}.

On the one hand, for the third and fifth terms, we write that
\begin{align*}
&
a_T\|\grad_T \cdot \bv_T\|_{\GDiss}\|\bv\cdot\bn\|_{\GDiss}
+
a_N\|\bv\cdot\bn\|_{\GDiss}\|\eps(\bv)\bn \cdot \bn\|_{\GDiss}
\\
&\leq
\sqrt{\frac{\alpha_T}{\alpha_N}}\sqrt{a_T}
\|\grad_T \cdot \bv_T\|_{\GDiss}\|\bv\|_{\AAA,\GDiss}
+
\sqrt{a_N}
\|\bv\|_{\AAA,\GDiss}\|\eps(\bv)\bn \cdot \bn\|_{\GDiss}
\\
&\leq
\max \left \{\sqrt{\frac{\alpha_T}{\alpha_N}}\sqrt{\alpha_T},\sqrt{\alpha_N}\right \}
(\mu_{\min}\rho_{\max})^{1/4}
\left (
\|\grad_T \cdot \bv_T\|_{\GDiss} + \|\eps(\bv)\bn \cdot \bn\|_{\GDiss}
\right )
\|\bv\|_{\AAA,\GDiss}
\\
&\leq
\max \left \{\sqrt{\frac{\alpha_T}{\alpha_N}}\sqrt{\alpha_T},\sqrt{\alpha_N}\right \}
\left (\frac{\rho_{\max}}{\mu_{\min}}\right )^{1/4}
\|\bv\|_{\AAA,\GDiss}
\|\eps(\bv)\|_{\mu,\GDiss}
\\
&=
\max \left \{\sqrt{\frac{\alpha_T}{\alpha_N}}\sqrt{\alpha_T},\sqrt{\alpha_N}\right \}
\vels^{-1/2}
\|\bv\|_{\AAA,\GDiss}
\|\eps(\bv)\|_{\mu,\GDiss}.
\end{align*}
On the other hand, we control the first term in the right-hand side
of \eqref{eq_Robin1} as follows:
\begin{equation*}
\|\AAA\bv\|_{\GDiss}\|\eps(\bv)\|_{\GDiss}
\leq
\sqrt{\frac{a_{\max}}{\mu_{\min}}}\|\bv\|_{\AAA,\GDiss}\|\eps(\bv)\|_{\mu,\GDiss}
=
\sqrt{\alpha_{\max}} \vels^{-1/2}\|\bv\|_{\AAA,\GDiss}\|\eps(\bv)\|_{\mu,\GDiss},
\end{equation*}
We can combine these last two estimates as
\begin{multline*}
2\|\AAA\bv\|_{\GDiss}\|\eps(\bv)\|_{\GDiss}
+
a_T\|\grad_T \cdot \bv_T\|_{\GDiss}\|\bv\cdot\bn\|_{\GDiss}
+
a_N\|\bv\cdot\bn\|_{\GDiss}\|\eps(\bv)\bn \cdot \bn\|_{\GDiss}
\\
\leq
\left (
\max \left \{\sqrt{\frac{\alpha_T}{\alpha_N}}\sqrt{\alpha_T},\sqrt{\alpha_N}\right \}
+
2\sqrt{\alpha_{\max}}
\right )
\vels^{-1/2}\|\bv\|_{\AAA,\GDiss}\|\eps(\bv)\|_{\mu,\GDiss}
\\
\leq
\left (
2 + \sqrt{\frac{\alpha_{\max}}{\alpha_{\min}}}
\right )
\sqrt{\alpha_{\max}}
\vels^{-1/2}\|\bv\|_{\AAA,\GDiss}\|\eps(\bv)\|_{\mu,\GDiss},
\end{multline*}
which concludes the proof.
\end{proof}

\begin{lemma}[Improved boundary estimate for the sphere]
\label{lemma_improved_sphere}
Assume that $\GDiss = \partial B_\RRR$, that $\bh = \bx$
in a neighborhood of $\GDiss$, and that $\AAA$ satisfies \eqref{eq_assumption_AAA}.
Then, the estimate given in \eqref{eq_estimate_LR_GDiss} holds true with
$\zeta \eq 0$.
\end{lemma}

\begin{proof}
We start by simplifying the identity given in \eqref{eq_identity_robin_boundary}
in Proposition \ref{proposition_identity_robin_boundary}. Indeed, since
here $\bn = \RRR^{-1} \bx$, and $\bh = \bx$, we have
\begin{align*}
(\AAA\bv,(\bh\cdot\grad)\bv)_{\GDiss}
&=
2(\AAA\bv,\bh\eps(\bv))_{\GDiss}
+
(\AAA\bv,(\tgrad \bh)\bv)_{\GDiss}
+
a_T((\bh\cdot\bn)(\grad_T \cdot \bv_T),\bv\cdot\bn)_{\GDiss}
\\
&-
a_N ((\bn\cdot\grad)\bh (\bv \cdot \bn),\bv)_{\GDiss}
-
a_N((\bh\cdot\bn)(\bv\cdot\bn),\eps(\bv)\bn \cdot \bn)_{\GDiss}.
\\
&=
2(\AAA\bv,\bh\eps(\bv))_{\GDiss}
+
\|\bv\|_{\AAA,\GDiss}^2
+
a_T\ell ((\grad_T \cdot \bv_T),\bv\cdot\bn)_{\GDiss}
\\
&-
a_N \|\bv \cdot \bn\|_{\GDiss}^2
-
a_N\ell (\bv\cdot\bn,\eps(\bv)\bn \cdot \bn)_{\GDiss}.
\end{align*}
As a result, compared to the general case, the two terms
associated with $\zeta$ in \eqref{eq_estimate_LR_GDiss}
disappear when taking the imaginary part:
\begin{align*}
\Im (\AAA\bv,(\bh\cdot\grad)\bv)_{\GDiss}
&=
2(\AAA\bv,\bh\eps(\bv))_{\GDiss}
\\
&+
a_T\ell ((\grad_T \cdot \bv_T),\bv\cdot\bn)_{\GDiss}
-
a_N\ell (\bv\cdot\bn,\eps(\bv)\bn \cdot \bn)_{\GDiss}.
\end{align*}
The remainder of the proof proceeds as in Lemma \ref{lemma_robin_boundary}.
\end{proof}

\section{Sharp stability estimates for simple Robin boundaries}
\label{section_simple_robin}

In this section, we consider the situation where $\bh \times \bn = \bzero$ on $\GDiss$
and $\AAA$ satisfies \eqref{eq_assumption_AAA}, which allows us to prove
sharp stability estimates. We also assume that the assumptions listed in
\eqref{eq_assumptions_multiplier} and \eqref{eq_assumption_sign_boundary}
hold true, and that we can take $m > 0$ in \eqref{eq_lower_bounds_h}.
Throughout this section, we also assume that $\uuu \in \BH^1_{\GDir}(\Omega)$
is a solution to \eqref{eq_helmholtz_weak}. Notice that then, $\uuu \in \BH^s(\Omega)$
for some $s \in (3/2,2]$ due to Theorem \ref{theorem_morawetz_identity}.

Since the bound we establish in this section is close from sharp, we take
care to make the constant fully-explicit. Although it makes the derivation
more intricate, having a fully-explicit constant it may be useful, e.g., in assessing
discretization errors in numerical approximations~\cite{chaumontfrelet_ern_vohralik_2021a}.

\begin{lemma}[Morawetz estimate]
\label{lemma_morawetz_estimate}
We have
\begin{multline}
\label{eq_morawetz_estimate}
2\frac{\gamma}{\MMM}\omega^2\|\uuu\|_{\rho,\Omega}^2
+
2\frac{\mmm}{\MMM}\RRR\|\eps(\uuu)\|_{\mu,\GDiss}^2
\leq
\\
\left \{
\frac{(d-2+\varepsilon)}{M} + \zeta + \frac{\ks}{\alpha_{\min}}
\right \}
\|\fff\|_{\rho,\Omega}\|\uuu\|_{\rho,\Omega} +
2\ks\omega^{-1}\|\fff\|_{\rho,\Omega}\sqrt{\mu_{\min}}\|\grad \uuu\|_\Omega
\\
+
\Crob \ks^{1/2}\omega^{1/2}\|\uuu\|_{\AAA,\GDiss}\RRR^{1/2}\|\eps(\uuu)\|_{\mu,\GDiss}.
\end{multline}
\end{lemma}

\begin{proof}
We have $\LB_{\bh,\GDir}(\uuu) \geq 0$ by assumption. On the other hand,
\eqref{eq_stress_identities} and \eqref{eq_lower_bounds_h} show that
\begin{equation*}
\LB_{\bh,\GDiss}(\uuu)
\geq
2\mmm\RRR\|\eps(\uuu)\|_{\mu,\GDiss}^2,
\end{equation*}
so we lower bound the left-hand side of \eqref{eq_morawetz_identity} as follows:
\begin{multline}
\label{eq_lower_bd_rhs2.4}
(d-\eta)\omega^2\|\uuu\|_{\rho,\Omega}^2
+
2\mmm\RRR\|\eps(\uuu)\|_{\mu,\GDiss}^2
\leq
\\
\omega^2 \int_\Omega (\div \bh + \LV_{\bh}(\rho))\rho|\uuu|^2
+
\LB_{\bh,\GDiss}(\uuu)
+
\LB_{\bh,\GDir}(\uuu).
\end{multline}

We then estimate the four terms in the right-hand side of \eqref{eq_morawetz_identity}.
First, we have
\begin{equation*}
2\Re (\rho\fff,(\bh\cdot\grad)\uuu)_\Omega
\leq
2\MMM\RRR\|\fff\|_{\rho,\Omega}\|\grad\uuu\|_{\rho,\Omega}
\leq
2\MMM\ks\omega^{-1}\|\fff\|_{\rho,\Omega}\sqrt{\mu_{\min}}\|\grad\uuu\|_\Omega.
\end{equation*}
Then, using the assumption on the multiplier $\bh$ in \eqref{eq_assumption_LR_Omega}
and the G\aa rding inequality recorded in \eqref{eq_garding_real}, we write that
\begin{align*}
\LR_{\bh,\Omega}(\uuu)
&\leq
(d-2+\varepsilon) \left \{
2\|\eps(\uuu)|_{\mu,\Omega}^2 + \|\div \uuu\|_{\lambda,\Omega}^2
\right \}
\\
&=
(d-2+\varepsilon)
\left \{
\Re (\rho\fff,\uuu)_\Omega + \omega^2\|\uuu\|_{\rho,\Omega}^2
\right \}
\\
&=
(d-2+\varepsilon)
\omega^2\|\uuu\|_{\rho,\Omega}^2
+
(d-2+\varepsilon)
\|\fff\|_{\rho,\Omega}\|\uuu\|_{\rho,\Omega}.
\end{align*}
For the third term in the right-hand side of \eqref{eq_morawetz_identity},
we simply use \eqref{eq_estimate_LR_GDiss} and \eqref{eq_garding_imag},
which gives
\begin{align*}
\LR_{\bh,\GDiss}(\uuu)
&\leq
M
\left \{
\Crob \ks^{1/2}\omega^{1/2}\|\uuu\|_{\AAA,\GDiss}\ell^{1/2}\|\eps(\uuu)\|_{\mu,\GDiss}
+
\zeta \omega\|\uuu\|_{\AAA,\GDiss}^2
\right \}
\\
&\leq
M
\Crob \ks^{1/2}\omega^{1/2}\|\uuu\|_{\AAA,\GDiss}\ell^{1/2}\|\eps(\uuu)\|_{\mu,\GDiss}
+
M\zeta \|\fff\|_{\rho,\Omega}\|\uuu\|_{\rho,\Omega}.
\end{align*}
Finally, recalling the upper bound on $\|\bh\|_{\BL^\infty(\Omega)}$
in \eqref{eq_lower_bounds_h} and the relation
$\rho_{\max} \leq a_{\min}/(\alpha_{\min} \vels)$
between $\rho$ and $\AAA$, we derive the estimate
\begin{equation*}
\omega^2 \int_{\GDiss}(\bh\cdot\bn)\rho|\uuu|^2
\leq
\omega^2\MMM \ell \|\uuu\|_{\rho,\GDiss}^2
\leq
\frac{M}{\alpha_{\min}} \ks \omega\|\uuu\|_{\AAA,\GDiss}^2
\leq
\frac{M}{\alpha_{\min}} \ks \|\fff\|_{\rho,\Omega}\|\uuu\|_{\rho,\Omega}
\end{equation*}
where we employed \eqref{eq_garding_imag} in the last inequality.

We may sum up the four above estimates as follows,
\begin{align*}
&2\Re (\rho\fff,(\bh\cdot\grad)\uuu)_\Omega
+
\omega^2 \int_{\GDiss} (\bh\cdot\bn)\rho|\uuu|^2
+
\LR_{\bh,\GDiss}(\uuu) + \LR_{\bh,\Omega}(\uuu)
\\
&\leq
2\MMM\ks\omega^{-1}\|\fff\|_{\rho,\Omega}\sqrt{\mu_{\min}}\|\grad \uuu\|_\Omega
\\
&+
(d-2+\varepsilon) \omega^2\|\uuu\|_{\rho,\Omega}^2
\\
&+
\left \{
(d-2+\varepsilon) + M\zeta + \frac{M}{\alpha_{\min}}\ks
\right \}
\|\fff\|_{\rho,\Omega}\|\uuu\|_{\rho,\Omega} +
\\
&+
M\Crob \ks^{1/2}\omega^{1/2}\|\uuu\|_{\AAA,\GDiss}\RRR^{1/2}\|\eps(\uuu)\|_{\mu,\GDiss},
\end{align*}
so that using \eqref{eq_morawetz_identity} and \eqref{eq_lower_bd_rhs2.4} and removing
$(d-2+\varepsilon)\omega^2\|\uuu\|_{\rho,\Omega}^2$ from
both sides yields
\begin{multline*}
(2-\varepsilon-\eta)\omega^2\|\uuu\|_{\rho,\Omega}^2
+
2\mmm\RRR\|\eps(\uuu)\|_{\mu,\GDiss}^2
\leq
\\
\left \{
(d-2+\varepsilon) + M\zeta + \frac{M}{\alpha_{\min}}\ks
\right \}
\|\fff\|_{\rho,\Omega}\|\uuu\|_{\rho,\Omega} +
2\MMM\ks\omega^{-1}\|\fff\|_{\rho,\Omega}\sqrt{\mu_{\min}}\|\grad \uuu\|_\Omega
\\
+
\MMM\Crob \ks^{1/2}\omega^{1/2}\|\uuu\|_{\AAA,\GDiss}\RRR^{1/2}\|\eps(\uuu)\|_{\mu,\GDiss}.
\end{multline*}
The estimate in \eqref{eq_morawetz_estimate} then follows using the definition
of $\gamma$ and dividing by $\MMM$.
\end{proof}

\begin{lemma}[Young inequalities]
The estimate
\begin{equation}
\label{eq_korn_rhs}
\begin{aligned}
2\ks \omega^{-1}\|\fff\|_{\rho,\Omega}\mu_{\min}^{1/2} \|\grad \uuu\|_\Omega
&\leq
\left (
\sqrt{2} (1+\chi\ks^{-1})\ks^{2-\theta}
+
2\alpha_N^{-1/2}\ks^{3/2-\tau}
\right )\omega^{-2}\|\fff\|_{\rho,\Omega}^2
\\
&+
(\sqrt{2} \ks^\theta + 2\ks)\|\fff\|_{\rho,\Omega}\|\uuu\|_{\rho,\Omega}
\\
&+
2\ks^\tau \omega^{1/2}\|\uuu\|_{\AAA,\GDiss}\RRR^{1/2}\|\eps(\uuu)\|_{\mu,\GDiss}.
\end{aligned}
\end{equation}
holds true for all $\theta,\tau \in \R$.
\end{lemma}

\begin{proof}
We start by combining the weighted Korn inequality we established
in \eqref{eq_weighted_korn_inequality} together with the G\aa rding inequality
in \eqref{eq_garding_real} and the boundary estimate in \eqref{eq_garding_imag}
to show that
\begin{align*}
\mu_{\min} \|\grad \uuu\|_\Omega^2
&\leq
(1+\chi\ks^{-1})\|\fff\|_{\rho,\Omega}\|\uuu\|_{\rho,\Omega}
\\
&+
\omega^2 \|\uuu\|_{\rho,\Omega}^2
+
2\alpha_N^{-1/2}\ks^{-1/2}\omega^{1/2}\|\uuu\|_{\AAA,\GDiss}\RRR^{1/2}\|\eps(\uuu)\|_{\mu,\GDiss}.
\end{align*}
We can therefore write that
\begin{align*}
\ks^2 \omega^{-2}\|\fff\|_{\rho,\Omega}^2\mu_{\min} \|\grad \uuu\|_\Omega^2
&\leq
(\ks^2+\chi\ks)\omega^{-2}\|\fff\|_{\rho,\Omega}^2\|\fff\|_{\rho,\Omega}\|\uuu\|_{\rho,\Omega}
\\
&+
\ks^2 \|\fff\|_{\rho,\Omega}^2\|\uuu\|_{\rho,\Omega}^2
+
2\alpha_N^{-1/2}\ks^{3/2}\omega^{-2}\|\fff\|_{\rho,\Omega}^2\omega^{1/2}\|\uuu\|_{\AAA,\GDiss}\RRR^{1/2}\|\eps(\uuu)\|_{\mu,\GDiss}.
\end{align*}
The next step is to employ Young inequality twice to obtain
\begin{align*}
&\ks^2 \omega^{-2}\|\fff\|_{\rho,\Omega}^2\mu_{\min} \|\grad \uuu\|_\Omega^2
\\
&\leq
\frac{1}{2\ks^{2\theta}} (\ks^2+\chi\ks)^2\omega^{-4}\|\fff\|_{\rho,\Omega}^4
+
\frac{\ks^{2\theta}}{2} \|\fff\|_{\rho,\Omega}^2\|\uuu\|_{\rho,\Omega}^2
+
\ks^2 \|\fff\|_{\rho,\Omega}^2\|\uuu\|_{\rho,\Omega}^2
\\
&+
\frac{\ks^3\omega^{-4}}{\alpha_N\ks^{2\tau}} \|\fff\|_{\rho,\Omega}^4
+
\ks^{2\tau} \omega\|\uuu\|_{\AAA,\GDiss}^2\RRR\|\eps(\uuu)\|_{\mu,\GDiss}^2.
\\
&=
\frac{1}{2} (1+\chi\ks^{-1})^2\ks^{4-2\theta}\omega^{-4}\|\fff\|_{\rho,\Omega}^4
+
\frac{\ks^{2\theta}}{2} \|\fff\|_{\rho,\Omega}^2\|\uuu\|_{\rho,\Omega}^2
+
\ks^2 \|\fff\|_{\rho,\Omega}^2\|\uuu\|_{\rho,\Omega}^2
\\
&+
\frac{\ks^{3-2\tau}\omega^{-4}}{\alpha_N} \|\fff\|_{\rho,\Omega}^4
+
\ks^{2\tau} \omega\|\uuu\|_{\AAA,\GDiss}^2\RRR\|\eps(\uuu)\|_{\mu,\GDiss}^2.
\end{align*}
We regroup the terms and multiply by $4$, giving us finally
\begin{align*}
4\ks^2 \omega^{-2}\|\fff\|_{\rho,\Omega}^2\mu_{\min} \|\grad \uuu\|_\Omega^2
&\leq
\left (
2 (1+\chi\ks^{-1})^2\ks^{4-2\theta}
+
4\alpha_N^{-1}\ks^{3-2\tau}
\right )\omega^{-4}\|\fff\|_{\rho,\Omega}^4
\\
&+
(2\ks^{2\theta} + 4\ks^2)\|\fff\|_{\rho,\Omega}^2\|\uuu\|_{\rho,\Omega}^2
\\
&+
4\ks^{2\tau} \omega\|\uuu\|_{\AAA,\GDiss}^2\RRR\|\eps(\uuu)\|_{\mu,\GDiss}^2.
\end{align*}
Now, \eqref{eq_korn_rhs} follows by taking the square root of both sides,
since $\sqrt{a^2 + b^2 + c^2} \leq a + b + c$.
\end{proof}

% The following quadratic inequality will be useful to establish our main result.

\begin{lemma}[Quadratic inequality]
Let $x \geq 0$ with $ax^2 \leq c + bx$ for $a,b,c > 0$. Then, 
\begin{equation}
\label{eq_quadratic_inequality}
ax \leq \sqrt{ac} + b.
\end{equation}
\end{lemma}

\begin{theorem}[Stability estimate for simple Robin boundaris]
\label{theorem_simple_robin}
Under the assumpions listed in the beginning of this section,
there exists a unique solution $\uuu \in \BH^1_{\GDir}(\Omega)$
to \eqref{eq_helmholtz_weak}, and we have
\begin{align}
\nonumber
\frac{\gamma}{\MMM}\omega^2\|\uuu\|_{\rho,\Omega}
\leq
\Big \{
&
\left (\frac{d-2+\varepsilon}{2M} + \frac{\zeta}{2}\right )
+
\sqrt{\chi}\ks^{1/6}
+
\frac{\MMM}{4m} \ks^{1/3}
+
\\
\nonumber
&
\left ( \frac{5}{2} + \frac{M}{4m}\Crob \right ) \ks^{2/3}
+
\\
\label{eq_stability_estimate}
&
\left (
1 + \frac{1}{2\alpha_{\min}} + \frac{\MMM}{16\mmm} \Crob^2
\right )\ks
\Big \}\|\fff\|_{\rho,\Omega}.
\end{align}
\end{theorem}

\begin{proof}
We start by plugging \eqref{eq_korn_rhs} into \eqref{eq_morawetz_estimate}, which provides
the estimate
\begin{multline}
\label{tmp_morawetz_estimate}
2\frac{\gamma}{\MMM}\omega^2\|\uuu\|_{\rho,\Omega}^2
+
2\frac{\mmm}{\MMM}\RRR\|\eps(\uuu)\|_{\mu,\GDiss}^2
\leq
\left (
\sqrt{2} (1+\chi\ks^{-1})\ks^{2-\theta} + 2\alpha_N^{-1/2}\ks^{3/2-\tau} \right )
\omega^{-2}\|\fff\|_{\rho,\Omega}^2
\\
\left \{
\left (\frac{d-2+\varepsilon}{M} + \zeta\right )
+
\sqrt{2}\ks^{\theta}
+
\left (\frac{1}{\alpha_{\min}}+2\right )\ks
\right \}
\|\fff\|_{\rho,\Omega}\|\uuu\|_{\rho,\Omega}
\\
+
(\Crob \ks^{1/2} + 2\ks^\tau)
\omega^{1/2}\|\uuu\|_{\AAA,\GDiss}\RRR^{1/2}\|\eps(\uuu)\|_{\mu,\GDiss}.
\end{multline}
The next step is to apply Young inequality in the last term, namely
\begin{multline}
\label{tmp_young_inequality}
(\Crob \ks^{1/2} + 2\ks^\tau)
\omega^{1/2}\|\uuu\|_{\AAA,\GDiss}\RRR^{1/2}\|\eps(\uuu)\|_{\mu,\GDiss}
\\
=
2\frac{1}{2\sqrt{2}}
(\Crob \ks^{1/2} + 2\ks^\tau)
\omega^{1/2}\|\uuu\|_{\AAA,\GDiss}\RRR^{1/2}\sqrt{2}\|\eps(\uuu)\|_{\mu,\GDiss}
\\
\leq
\frac{1}{8}\frac{\MMM}{\mmm}
(\Crob \ks^{1/2} + 2\ks^\tau)^2
\omega\|\uuu\|_{\AAA,\GDiss}^2
+
2\frac{\mmm}{\MMM}\RRR\|\eps(\uuu)\|_{\mu,\GDiss}^2.
\end{multline}
We then observe that the last terms in the left-hand side of
\eqref{tmp_morawetz_estimate} and in the right-hand side of
\eqref{tmp_young_inequality} coincide, so that
\begin{multline*}
2\frac{\gamma}{\MMM}\omega^2\|\uuu\|_{\rho,\Omega}^2
\leq
\left (\sqrt{2} (1+\chi\ks^{-1})\ks^{2-\theta} + 2\ks^{3/2-\tau} \right )
\omega^{-2}\|\fff\|_{\rho,\Omega}^2
\\
+
\left \{
\left (\frac{d-2+\varepsilon}{M} + \zeta\right )
+
\sqrt{2}\ks^\theta
+
\left (\frac{1}{\alpha_{\min}}+2\right )\ks
\right \}
\|\fff\|_{\rho,\Omega}\|\uuu\|_{\rho,\Omega}
\\
+
\frac{1}{8}\frac{\MMM}{\mmm}
(\Crob \ks^{1/2} + 2\ks^\tau)^2
\omega\|\uuu\|_{\AAA,\GDiss}^2.
\end{multline*}
We may then employ \eqref{eq_garding_imag}, leading to
\begin{multline*}
2\frac{\gamma}{\MMM}\omega^2\|\uuu\|_{\rho,\Omega}^2
\leq
\left (\sqrt{2} (1+\chi\ks^{-1})\ks^{2-\theta} + 2\ks^{3/2-\tau} \right )
\omega^{-2}\|\fff\|_{\rho,\Omega}^2
+
\\
\left \{
\left (\frac{d-2+\varepsilon}{M} + \zeta\right )
+
\sqrt{2}\ks^\theta
+
\left (\frac{1}{\alpha_{\min}}+2\right )\ks
+
\frac{1}{8}\frac{\MMM}{\mmm}
(\Crob \ks^{1/2} + 2\ks^\tau)^2
\right \}
\|\fff\|_{\rho,\Omega}\|\uuu\|_{\rho,\Omega},
\end{multline*}
and \eqref{eq_quadratic_inequality} then reveals that
\begin{multline}
\label{tmp_stability_estimate}
2\frac{\gamma}{\MMM}\omega^2\|\uuu\|_{\rho,\Omega}
\leq
\sqrt{
2\frac{\gamma}{M}
\left (\sqrt{2} (1+\chi\ks^{-1})\ks^{2-\theta} + 2\ks^{3/2-\tau} \right )}
\|\fff\|_{\rho,\Omega} +
\\
\left \{
\left (\frac{d-2+\varepsilon}{M} + \zeta\right )
+
\sqrt{2}\ks^{\theta}
+
\left (\frac{1}{\alpha_{\min}}+2\right )\ks
+
\frac{1}{8}\frac{\MMM}{\mmm}
(\Crob \ks^{1/2} + 2\ks^{\tau})^2
\right \}
\|\fff\|_{\rho,\Omega}.
\end{multline}

The remainder of the proof consists in simplifying the right-hand side
of \eqref{tmp_stability_estimate} and selecting the values of $\theta$
and $\tau$ so as to optimize the resulting powers of $\ks$. We start
with two simple estimates, namely
\begin{equation*}
\sqrt{2\frac{\gamma}{M}
\left (\sqrt{2} (1+\chi\ks^{-1})\ks^{2-\theta} + 2\ks^{3/2-\tau} \right )}
\leq
\sqrt{2\frac{\gamma}{M}}
\left (2^{1/4} (1+\chi\ks^{-1})^{1/2}\ks^{1-\theta/2} + \sqrt{2}\ks^{3/4-\tau/2} \right )
\end{equation*}
and
\begin{equation*}
(\Crob \ks^{1/2} + 2\ks^{\tau})^2
=
\Crob^2 \ks + 4\Crob \ks^{1/2+\tau} + 4\ks^{2\tau}.
\end{equation*}

At that point, to fix the values of $\theta$ and $\tau$, we equate the two
higher powers of $\ks$ respectively involving $\theta$ and $\tau$. More precisely,
we pick $\theta$ and $\tau$ such that $1-\theta/2 = \theta$ and $1/2+\tau = 3/4-\tau/2$.
The desired values are then $\theta_\star \eq 2/3$ and $\tau_\star \eq 1/6$, and we have
$\theta_\star = 1-\theta_\star/2 = 2/3$, $3/4-\tau_\star/2 = 1/2+\tau_\star = 2/3$,
$2\tau_\star = 1/3$.

This choice provides the estimate
\begin{align*}
2\frac{\gamma}{\MMM}\omega^2\|\uuu\|_{\rho,\Omega}
\leq
\Big \{
&
\left (\frac{d-2+\varepsilon}{M} + \zeta\right ) + \frac{\MMM}{2m} \ks^{1/3} +
\\
&
\left [
\sqrt{2}+\frac{\MMM}{2m}\Crob +
\sqrt{2\frac{\gamma}{M}}
\left (2^{1/4} (1+\chi\ks^{-1})^{1/2}+ \sqrt{2} \right )
\right ]
\ks^{2/3} +
\\
&
\left (
2 + \frac{1}{\alpha_{\min}} + \frac{\MMM}{8\mmm} \Crob^2
\right )\ks
\Big \}\|\fff\|_{\rho,\Omega}.
\end{align*}

The remainder of the proof consists in simplifying the $\ks^{2/3}$ term.
Specifically, recalling that $\gamma/M \leq 1$ and $(1+x)^{1/2} \leq 1+x^{1/2}$,
the estimate follows from
\begin{align*}
&
\sqrt{2}+\frac{\MMM}{2m}\Crob
+
\sqrt{2\frac{\gamma}{M}}
\left (2^{1/4} (1+\chi\ks^{-1})^{1/2}+ \sqrt{2} \right )
\\
&\leq
\sqrt{2}+\frac{\MMM}{2m}\Crob
+
\sqrt{2} \left (2^{1/4} (1+\sqrt{\chi}\ks^{-1/2}) + \sqrt{2} \right )
\\
&\leq
2\sqrt{2}+\frac{\MMM}{2m}\Crob + 2 (1+\sqrt{\chi}\ks^{-1/2})
\\
&\leq
5+\frac{\MMM}{2m}\Crob + 2\sqrt{\chi}\ks^{-1/2},
\end{align*}
where we used the fact that $2\sqrt{2} \leq 3$. This leads us to
\begin{align*}
2\frac{\gamma}{\MMM}\omega^2\|\uuu\|_{\rho,\Omega}
\leq
\Big \{
&
\left (\frac{d-2+\varepsilon}{M} + \zeta\right )
+
2\sqrt{\chi}\ks^{1/6}
+
\frac{\MMM}{2m} \ks^{1/3}
+
\\
&
\left ( 5 + \frac{M}{2m}\Crob \right ) \ks^{2/3}
+
\\
&
\left (
2 + \frac{1}{\alpha_{\min}} + \frac{\MMM}{8\mmm} \Crob^2
\right )\ks
\Big \}\|\fff\|_{\rho,\Omega},
\end{align*}
and \eqref{eq_stability_estimate} follows after dividing by two.
\end{proof}

The strength of Theorem \ref{theorem_simple_robin} is that the frequency
scaling in the stability estimate is sharp. A drawback, however, is that
the constraint that $\bh \times \bn = \bzero$ on $\GDiss$ forces the boundary
to have a rather simple shape. Hereafeter, we allow for more generality in the
shape of $\GDiss$, but at the price of less accurate stability bounds.

\section{Suboptimal stability estimates for complex Robin boundaries}
\label{section_general_robin}

We now revisit the arguments of \cite{cummings_feng_2006a} allowing for
more general shapes for $\GDiss$, but at the price of less precise
stability results. In this setting, the estimates of Section
\ref{section_new} are not available, and the challenge is to control
the boundary term $\LR_{\bh,\GDiss}(\uuu)$. Following \cite[Lemma 10]{cummings_feng_2006a},
the key idea is to control $\|\grad \uuu\|_{\GDiss}$ by the $\BH^2(\Omega)$ norm of $\uuu$, which
is itself controlled by $\|\fff\|_{\rho,\Omega}$ and the $\BH^1(\Omega)$ norm of $\uuu$ due to
elliptic regularity. Here, for the sake of generality, we state the key regularity result
needed as an assumption. We will show in Appendix \ref{appendix_elliptic_regularity} that
this assumption does hold true with a constant $\Creg$ independent of $\omega$
if $\GDiss$ is smooth and the coefficients are smooth in a neighborhood
of $\GDiss$. We will also show that $\Creg$ is independent of $\lambda/\mu$,
when both $\GDiss$ and $\GDir$ are smooth and $\lambda$ and $\mu$ are constant.

\begin{assumption}[Elliptic regularity]
\label{assumption_elliptic_regularity}
There exists a constant $\Creg$
such that, if $\bw \in \BH^1_{\GDir}(\Omega)$ solves
\begin{equation*}
(\sig(\bw),\eps(\bv))_\Omega
=
\omega^2 (\rho\bg,\bv)_\Omega
+
\omega (\AAA\bk,\bv)_{\GDiss}
\end{equation*}
for all $\bv \in \BH^1_{\GDir}(\Omega)$ with
$\bg \in \BL^2(\Omega)$ and $\bk \in \BH^1(\Omega)$,
then $\grad \bw \in \BL^2(\GDiss)$ with
\begin{equation*}
\ell \|\grad \bw\|_{\mu,\GDiss}^2
\leq
\Creg
(1+\ks)
\left \{
\omega^2\|\bg\|_{\rho,\Omega}
+
\omega^2\|\bk\|_{\rho,\Omega} + \|\eps(\bk)\|_{\mu,\Omega}^2
+
\|\eps(\bw)\|_{\mu,\Omega}^2
\right \}
\end{equation*}
\end{assumption}

In the remainder of this Section, we assume that $\GDiss = \partial \CO$ and
$\GDir = \partial D$ with $D \subset\!\subset \CO$ so that $\Omega = \CO \setminus D$.
Recall that $\hGDiss \eq (1/\RRR)\GDiss$ is the rescale absorbing boundary.
Then, we have the following Korn inequality:

\begin{proposition}[Korn inequality]
For all $\bv \in \BH^1_{\GDir}(\Omega)$, we have
\begin{equation}
\label{eq_korn_general}
\|\tgrad \bv\|_{\mu,\Omega}^2
\leq
\LK^2
(1+\ks^{-1})^2
\left (
\omega^2\|\bv\|_{\rho,\Omega}^2 + \|\eps(\bv)\|_{\mu,\Omega}^2
\right )
\qquad
\forall \bv \in \BH^1_{\GDir}(\Omega),
\end{equation}
where the constant $\LK$ only depends on $\hGDiss$,
$\rho_{\max}/\rho_{\min}$, and $\mu_{\max}/\mu_{\min}$.
\end{proposition}

\begin{proof}
We first observe that if $\bv \in \BH^1_{\GDir}(\Omega)$, its extension
by zero in $D$, $\tbv$, belongs to $\BH^1(\CO)$.
Then, we employ a scaling argument, and introduce the function
$\hbv(\hbx)=\tbv(\RRR\hbx)$ for $\hbx \in \hCO \eq (1/\RRR)\CO$.

It is clear that $\hbv \in H^1(\hCO)$, and thus, we can apply the
standard Korn's inequality in $\hCO$, see, e.g. \cite[Theorem 11.2.16]{brenner_scott_2008a}
for $d=2$ and \cite[(1.2.35]{ciarlet_2002a} for $d=3$, to find
\[
\|\tgrad \hat \bv\|_{\hCO}^2
\leq
(\widehat{\LK})^2
\left (\|\hat\bv\|_{\hCO}^2 + \|\eps(\hat\bv)\|_{\hCO}^2\right).
\]
for some positive constant $\widehat{\LK}$ depending only on $\hCO$,
and hence on $\hGDiss$. Due to the definition of $\hbv$, we also have
\begin{equation*}
\|\tgrad \hbv\|_{\hCO}^2 = \RRR^{2-d} \|\tgrad \tbv\|_{\CO}^2 = \RRR^{2-d} \|\tgrad \bv\|_{\Omega}^2
\qquad
\|\hbv\|_{\hCO}^2 = \RRR^{-d} \|\tbv\|_{\CO}^2 = \RRR^{-d} \|\bv\|_{\Omega}^2,
\end{equation*}
leading to
\begin{equation*}
\|\tgrad \bv\|_{\Omega}^2
\leq
(\widehat{\LK})^2
\left (\ell^{-2}\|\bv\|_{\Omega}^2 + \|\eps(\bv)\|_{\Omega}^2\right ).
\end{equation*}
This proves \eqref{eq_korn_general} since $\ell^{-1}\leq \omega/(\vels \ks)$.
\end{proof}

In practice, the constants $\Creg$ and $\LK$ are hard to estimate explicitly.
Besides, because the estimates obtained in this section seem to be
suboptimal (the frequency scaling is quadratic instead of linear),
it is unlikely that a fully explicit bound expressed in terms of
$\Creg$ will be useful. As a result, in the remainder of this section,
we let $C$ denote a generic constant that only depends on $\Creg$, $\LK$,
$\alpha_{\min}$ and $\alpha_{\max}$, $M$ and $\gamma$. If $\xi$ is a real
number, we use the notation $C_\xi$ to indicate that this constant may additionally
depend on $\xi$.

Notice that then, the estimates derived in this section are only meaningful
if $\Creg$ is independent of $\omega$ and $\ell$. Using standard elliptic
regularity results, we will show in Appendix \ref{appendix_elliptic_regularity}
that this is indeed true, at least under reasonable assumptions. We will also
show that $\Creg$ is independent of $\lambda_{\max}/\mu_{\min}$, at least when
$\lambda$ and $\mu$ are constant.

\begin{lemma}[Dissipative boundary term]
For all $\xi \in (0,1)$, we have
\label{lemma_coarse_bound_RGDiss}
\begin{equation}
\label{eq_robin_suboptimal}
|\LR_{\bh,\GDiss}(\uuu)|
\leq
C_\xi
(1+\ks^2) \omega\|\uuu\|_{\AAA,\GDiss}^2
+
\omega^{-2}\|\fff\|_{\rho,\Omega}^2
+
\xi \omega^2\|\uuu\|_{\rho,\Omega}^2.
\end{equation}
\end{lemma}

\begin{proof}
Recall that since $\sig(\uuu)\bn = i\omega\AAA\uuu$ on $\GDiss$, we have
\begin{equation*}
\LR_{\bh,\GDiss}(\uuu)
\eq
2 \Re \int_{\GDiss}
(\sig(\uuu)\bn) \cdot ((\bh \cdot \grad) \overline{\uuu}).
=
2 \Re i\omega\int_{\GDiss}
\AAA\uuu \cdot ((\bh \cdot \grad) \overline{\uuu}).
\end{equation*}
Recalling that $\|{\cdot}\|_{\AAA,\GDiss} \leq C \vels^{-1/2}\|{\cdot}\|_{\mu,\GDiss}$,
and using that \eqref{eq_upper_bounds_h} implies $\|\bh\|_{\BL^\infty(\Omega)}\leq \MMM \RRR$,
we have
\begin{align*}
|\LR_{\bh,\GDiss}(\uuu)|
&\leq
C\omega^{1/2}\|\uuu\|_{\AAA,\GDiss}
(\omega/\vels)^{1/2}
\|(\bh\cdot\grad) \uuu\|_{\mu,\GDiss}
\\
&\leq
C\ks^{1/2}  \omega^{1/2}
\|\uuu\|_{\AAA,\GDiss}
\ell^{1/2}
\|\grad \uuu\|_{\mu,\GDiss}.
\end{align*}
We conclude the proof with Young's inequality
and Assumption \ref{assumption_elliptic_regularity}
as follows. For all $\tau \in (0,3)$, we have
\begin{align*}
|\LR_{\bh,\GDiss}(\uuu)|
&\leq
C_\tau (1+\ks) \ks \omega\|\uuu\|_{\AAA,\GDiss}^2
+
\tau \frac{\ell}{1+\ks}\|\grad \uuu\|_{\mu,\GDiss}^2
\\
&\leq
C_\tau (1+\ks) \ks \omega\|\uuu\|_{\AAA,\GDiss}^2
+
\tau
\left (
\omega^{-2}\|\fff\|_{\rho,\Omega}^2
+
\omega^2\|\uuu\|_{\rho,\Omega}^2
+
\|\eps(\uuu)\|_{\mu,\Omega}^2
\right ).
\end{align*}
Since, the G\aa rding identity in \eqref{eq_garding_real} gives
\begin{equation*}
\|\eps(\uuu)\|_{\mu,\Omega}^2
=
\Re(\rho\fff,\uuu)_\Omega + \omega^2\|\uuu\|_{\rho,\Omega}^2
\leq
\omega^{-2} \|\fff\|_{\rho,\Omega} + 2\omega^2\|\uuu\|_{\rho,\Omega}^2,
\end{equation*}
we have
\begin{align*}
|\LR_{\bh,\GDiss}^2(\uuu)|
&\leq
C_\tau (1+\ks^2) \omega\|\uuu\|_{\AAA,\GDiss}^2
+
\tau
\left (
2\omega^{-2}\|\fff\|_{\rho,\Omega}^2
+
3\omega^2\|\uuu\|_{\rho,\Omega}^2
\right )
\\
&\leq
C_\xi
(1+\ks^2) \omega\|\uuu\|_{\AAA,\GDiss}^2
+
\omega^{-2}\|\fff\|_{\rho,\Omega}^2
+
\xi \omega^2\|\uuu\|_{\rho,\Omega}^2
\end{align*}
upon taking $\xi = \tau/3$.
\end{proof}

\begin{theorem}[Stability estimate for general Robin boundary]
\label{theorem_general_robin}
We have
\begin{equation}
\label{eq_general_robin}
\omega^2\|\uuu\|_{\rho,\Omega}
\leq
C (1 + \ks^2) \|\fff\|_{\rho,\Omega}.
\end{equation}
\end{theorem}

\begin{proof}
The proof essentially proceeds as in Lemma \ref{lemma_morawetz_estimate}
with a different treatment of the boundary terms on $\GDiss$.
We start with the Morawetz identity in \eqref{eq_morawetz_identity}.
Recalling the assumptions in \eqref{eq_assumptions_multiplier} and the
notation in \eqref{eq_bounds_h}, we have the lower bound
\begin{equation*}
(d-\eta) \omega^2 \|\uuu\|_{\rho,\Omega}^2
\leq
\omega^2 \int_\Omega (\div \bh +\LV_{\bh}(\rho)) \rho|\uuu|^2
+
\LB_{\bh,\GDiss} + \LB_{\bh,\GDir}
\end{equation*}
for the left-hand side of \eqref{eq_morawetz_identity}.

We then upper bound the right-hand side of \eqref{eq_morawetz_identity}.
For the terms related to $\GDiss$, we recall from Lemma \ref{lemma_coarse_bound_RGDiss}
above that
\begin{equation*}
|\LR_{\bh,\GDiss}(\uuu)|
\leq
C_\xi
(1+\ks^2) \omega\|\uuu\|_{\AAA,\GDiss}^2
+
\omega^{-2}\|\fff\|_{\rho,\Omega}^2
+
\xi \omega^2\|\uuu\|_{\rho,\Omega}^2
\end{equation*}
for all $\xi \in (0,1)$, and we write that
\begin{equation*}
\omega^2
\int_{\GDiss} (\bh\cdot\bn)
\rho|\uuu|^2
\leq
\omega^2 M\RRR \|\uuu\|_{\rho,\Omega}^2
\leq
C \ks \omega \|\uuu\|_{\AAA,\Omega}^2.
\end{equation*}
For the remaining terms in the right-hand side
of \eqref{eq_morawetz_identity}, we closely follow
Lemma \ref{lemma_morawetz_estimate}. Here, we first
note that due to the G\aa rding inequality in \eqref{eq_garding_real},
we have
\begin{equation*}
\|\eps(\uuu)\|_{\mu,\Omega}
\leq
C
\left \{
\omega^{-1}\|\fff\|_{\rho,\Omega}
+
\omega \|\uuu\|_{\rho,\Omega}
\right \}.
\end{equation*}
Then, using Korn inequality in $\BH^1_{\GDir}(\Omega)$
and the definition of $\ks$, we have
\begin{align*}
2\Re (\rho\fff,(\bh\cdot\grad)\uuu)_\Omega
&\leq
C \ks\omega^{-1} \|\fff\|_{\rho,\Omega} \|\grad \uuu\|_{\mu,\Omega}
\\
&\leq
C (1+\ks) \omega^{-1} \|\fff\|_{\rho,\Omega}
\left (
\omega \|\uuu\|_{\rho,\Omega}
+
\|\eps(\uuu)\|_{\mu,\Omega}
\right )
\end{align*}
and therefore
\begin{equation*}
2\Re (\rho\fff,(\bh\cdot\grad)\uuu)_\Omega
\leq
C (1+\ks) \left (
\omega^{-2} \|\fff\|_{\rho,\Omega}^2
+
\omega^2 \|\uuu\|_{\rho,\Omega}^2
\right ).
\end{equation*}
Finally, since $\varepsilon \leq 2$ by definition,
\begin{align*}
\LR_{\bh,\Omega}(\uuu)
&\leq
(d-2+\varepsilon)\omega^2\|\uuu\|_{\rho,\Omega}^2
+
(d-2+\varepsilon)\|\fff\|_{\rho,\Omega}\|\uuu\|_{\rho,\Omega}
\\
&\leq
(d-2+\varepsilon)\omega^2\|\uuu\|_{\rho,\Omega}^2
+
C\|\fff\|_{\rho,\Omega}\|\uuu\|_{\rho,\Omega}.
\end{align*}
Combining those bounds, we have
\begin{align*}
&2\Re (\rho\fff,(\bh\cdot\grad)\uuu)_\Omega
+
\omega^2
\int_{\GDiss} (\bh\cdot\bn)
\rho|\uuu|^2
+
\LR_{\bh,\GDiss}(\uuu)
+
\LR_{\bh,\Omega}(\uuu)
\\
&\leq
C(1+\ks)\left \{\omega^{-2}\|\fff\|_{\rho,\Omega}^2
+
\|\fff\|_{\rho,\Omega}\|\uuu\|_{\rho,\Omega}
\right \}
+
C_\xi(1+\ks^2)\omega\|\uuu\|_{\AAA,\GDiss}^2
\\
&+
(d-2+\varepsilon+\xi)\omega^2\|\uuu\|_{\rho,\Omega}^2
\end{align*}
and if we substract $(d-2+\varepsilon+\xi)\omega^2\|\uuu\|_{\rho,\Omega}^2$
from both sides of \eqref{eq_morawetz_identity}, we obtain
\begin{align*}
(2-\eta-\varepsilon-\xi)\omega^2\|\uuu\|_{\rho,\Omega}^2
&\leq
C(1+\ks)\left \{\omega^{-2}\|\fff\|_{\rho,\Omega}^2
+
\|\fff\|_{\rho,\Omega}\|\uuu\|_{\rho,\Omega}
\right \}
\\
&+
C_\xi(1+\ks^2)\omega\|\uuu\|_{\AAA,\GDiss}^2.
\end{align*}
Since $2\gamma = 2-\eta-\varepsilon > 0$, we may chose $\xi = \gamma$, which gives
\begin{equation*}
\gamma \omega^2\|\uuu\|_{\rho,\Omega}^2
\leq
C(1+\ks)\left \{\omega^{-2}\|\fff\|_{\rho,\Omega}^2
+
\|\fff\|_{\rho,\Omega}\|\uuu\|_{\rho,\Omega}
\right \}
+
C_\gamma(1+\ks^2)\omega\|\uuu\|_{\AAA,\GDiss}^2.
\end{equation*}
Recalling \eqref{eq_garding_imag}, we have
$\omega\|\uuu\|_{\AAA,\GDiss}^2 \leq \|\fff\|_{\rho,\Omega}\|\uuu\|_{\rho,\Omega}$. Hence,
since $1+\ks \leq 2(1+\ks^2)$, and because the generic constants $C$ are allowed to depend on
$\gamma$, we can write
\begin{equation*}
\omega^2\|\uuu\|_{\rho,\Omega}^2
\leq
C(1+\ks^2) \left \{ \omega^{-2} \|\fff\|_{\rho,\Omega}^2
+
\|\fff\|_{\rho,\Omega}\|\uuu\|_{\rho,\Omega}\right \}.
\end{equation*}
At that point, we easily conclude using a Young inequality.
Indeed, for all $\tau>0$, we have
\begin{equation*}
(1+\ks^2)\|\fff\|_{\rho,\Omega}\|\uuu\|_{\rho,\Omega}
\leq
C_{\tau} (1+\ks^2)^2\omega^{-2}\|\fff\|_{\rho,\Omega}^2
+
\tau\omega^2\|\uuu\|_{\rho,\Omega}^2,
\end{equation*}
and taking $\tau$ small enough that
\begin{equation*}
\omega^2\|\uuu\|_{\rho,\Omega}^2
\leq
C(1+\ks^2)^2 \omega^{-2}\|\fff\|_{\rho,\Omega}^2,
\end{equation*}
where we used the fact that $(1+\ks^2) \leq (1+\ks^2)^2$.
\end{proof}

\section{Examples}
\label{section_examples}

In this section, we focus on the case where $\Omega \subset B_\RRR$
and $\GDiss \eq \partial B_\RRR$, and assume that $\AAA$ satisfies
\eqref{eq_assumption_AAA}. We will consider the multiplier $\bh(\bx) = \bx$.
As a result, we have in particular
\begin{equation*}
M = m = 1
\end{equation*}
in \eqref{eq_bounds_h}.
In addition, due to Lemmas \ref{lemma_robin_boundary} and \ref{lemma_improved_sphere},
because $\GDiss$ is a sphere, we can take
\begin{equation*}
\zeta = 0 \qquad \chi = \max \left \{ \frac{1}{\alpha_T}, \frac{d-1}{\alpha_N} \right \}
\end{equation*}
in \eqref{eq_estimate_LR_GDiss}. Finally, due to our choice of $\bh$, we can further simplify
$\LR_{\bh,\Omega}(\bv)$. Recall that by definition
\begin{align*}
\LR_{\bh,\Omega}(\bv)
&\eq
\int_\Omega \left \{
(\div \bh + \LV_{\bh}(\mu))2\mu|\eps(\bv)|^2
+
(\div \bh + \LV_{\bh}(\lambda))\lambda|\div \bv|^2
\right \}
\\
&-2\Re\int_\Omega \left \{
\mu \eps(\bv) : (\tgrad \bh \tgrad \overline{\bv})
+
\lambda (\div \bv) \tgrad \bh : (\tgrad \overline{\bv})^T
\right \}.
\end{align*}
Here, $\tgrad \bh = \tens{I}$, and we have on the one hand
\begin{equation*}
\int_\Omega \mu \eps(\bv) : (\tgrad \bh \tgrad \overline{\bv})
=
\int_\Omega \mu \eps(\bv) : \tgrad \overline{\bv}
=
\int_\Omega \mu |\eps(\bv)|^2
\end{equation*}
by symmetry of $\eps(\bv)$. On the other hand we have
\begin{equation*}
\int_\Omega \lambda (\div \bv) \tgrad \bh : (\tgrad \overline{\bv})^T
=
\int_\Omega \lambda (\div \bv) \tens{I} : (\tgrad \overline{\bv})^T
=
\int_\Omega \lambda |\div \bv|^2,
\end{equation*}
and recalling that $\div \bh = d$, it follows that
\begin{equation}
\label{eq_LR_Omega_bh_bx}
\LR_{\bh,\Omega}(\bv)
=
\int_\Omega \left \{
(d-2 + \LV_{\bh}(\mu))2\mu|\eps(\bv)|^2
+
(d-2 + \LV_{\bh}(\lambda))\lambda|\div \bv|^2
\right \}.
\end{equation}

\subsection{Star-shaped obstacles}

We start with the case where $\Omega = B_\RRR \setminus D$,
where $D \subset B_\RRR$ is a closed set star-shaped with respect to $\bo$.
The medium of propagation is assumed to be homogeneous,
meaning that $\rho,\mu$ and $\lambda$ are constant.
It is then clear that
$\LV_{\bh}(\rho) = \LV_{\bh}(\mu) = \LV_{\bh}(\lambda) = 0$,
so that we can further simplify the expression
of $\LR_{\bh,\Omega}$ given in \eqref{eq_LR_Omega_bh_bx} into
\begin{equation*}
\LR_{\bh,\Omega}^2(\bv)
=
(d-2)
\left (
2\|\eps(\bv)\|_{\mu,\Omega}^2
+
\|\div \bv\|_{\lambda,\Omega}^2
\right ).
\end{equation*}
For the same reason, we also have
\begin{equation*}
\int_\Omega(\div \bh + \LV_{\bh}(\rho)) \rho|\bv|^2
=
d\|\bv\|_{\rho,\Omega}^2.
\end{equation*}
It means that \eqref{eq_assumption_L2_Omega} and \eqref{eq_assumption_LR_Omega}
are satisfied with $\varepsilon = \eta = 0$. In other word, the assumption
in \eqref{eq_assumption_gamma} holds true with $\gamma = 1$.
As a result, we can simplify the stability estimate given in
\eqref{eq_stability_estimate} as follows:
\begin{align}
\nonumber
\omega^2\|\uuu\|_{\rho,\Omega}
\leq
\Big \{
&
\left (\frac{d}{2}-1\right )
+
\frac{1}{\sqrt{\alpha_{\min}}}\ks^{1/6}
+
\frac{1}{4} \ks^{1/3}
+
\\
\nonumber
&
\left ( \frac{5}{2} + \frac{1}{4}\Crob \right ) \ks^{2/3}
+
\\
&
\left (
1 + \frac{1}{2\alpha_{\min}} + \frac{1}{16} \Crob^2
\right )\ks
\Big \}\|\fff\|_{\rho,\Omega}.
\label{eq_simplified_morawetz_estimate_star_shaped}
\end{align}

We start with the ideal case where $\AAA = (\rho\mu)^{1/2} \tens{I}$,
whence $\alpha_N = \alpha_T = 1$.

\begin{corollary}[Star-shaped obstacle with favorable Robin coefficients]
\label{corollary_obstacle_ideal}
\begin{subequations}
\label{eq_obstacle_ideal}
Assume that $\alpha_N = \alpha_T = 1$. Then, we have
\begin{equation}
\label{eq_estimate_obstacle_ideal}
\omega^2\|\uuu\|_{\rho,\Omega}
\leq
\left \{
\left (\frac{d}{2}-1\right ) + \ks^{1/6} + \frac{1}{4} \ks^{1/3}
+
\frac{13}{4} \ks^{2/3}
+
\frac{33}{16} \ks
\right \}\|\fff\|_{\rho,\Omega}.
\end{equation}
In addition, the simplified estimate
\begin{equation}
\label{eq_estimate_obstacle_ideal_simplified}
\omega^2\|\uuu\|_{\rho,\Omega}
\leq
(3 + 5 \ks)
\|\fff\|_{\rho,\Omega}
% 3(1 + \ks)
\end{equation}
holds true.
\end{subequations}
\end{corollary}

%\todo{Je trouve bien un coefficient 5 devant $\ks$ pour l'estimation simplifi\'ee.}

\begin{proof}
Since $\alpha_N = \alpha_T = 1$ by assumption, we have $\Crob = 3$,
and \eqref{eq_simplified_morawetz_estimate_star_shaped} further simplifies
into
\begin{equation*}
\omega^2\|\uuu\|_{\rho,\Omega}
\leq
\Big \{
\left (\frac{d}{2}-1\right ) + \ks^{1/6} + \frac{1}{4} \ks^{1/3}
+
\left ( \frac{5}{2} + \frac{3}{4} \right ) \ks^{2/3}
+
\left (
1 + \frac{1}{2} + \frac{9}{16}
\right )\ks
\Big \}\|\fff\|_{\rho,\Omega}.
\end{equation*}
Then, \eqref{eq_estimate_obstacle_ideal} simply follows by simplifying the fractions, whereas
we obtain \eqref{eq_estimate_obstacle_ideal_simplified} by noticing that $d \leq 3$ and
estimating the middle terms with the inequality ``$b \leq 1/p + b^q/q$'' 
with $p,q>1$ such that $1/p+1/q=1$ and $b\geq0$, yielding
\[
\ks^{1/6}\leq 5/6+\ks/6, \quad \ks^{1/3}\leq 2/3+\ks/3, \quad \ks^{2/3}\leq 1/3+2\ks/3.
\]
\end{proof}

The estimates in Corollary \ref{corollary_obstacle_ideal} have two desirable properties.
First, the coefficient multiplying $\ks$ in the leading term is just slightly above two.
The sharp estimate established in \cite{galkowski_spence_wunsch_2020a} and discussed at
\eqref{eq_sharp_acoustic_intro} in the introduction has coefficient $1$, so our bound is
sharp by a factor $33/16 \sim 2$. Second, the estimate is fully explicit and does not depend
on $\lambda$, as expected.

Unfortunately, the choice of coefficients for the Robin matrix $\AAA$ on the dissipative
boundary is not realistic from a numerical approximation perspective. Indeed, the coefficients
are correctly tuned to absorb shear waves, but are not suited for pressure waves. A better choice
is then $\alpha_T = 1$ and $\alpha_N = \velpV/\velsV$, see \cite[(54)]{gachter_grote_2003a} in the
case of constant coefficients $\lambda, \mu$ and $\rho$. The tangential component of the matrix
then correctly absorbed shear waves: if a shear wave reaches the boundary at normal incidence,
then the oscillations happen in the tangential directions. The modification of the coefficients
$\alpha_N$ makes sure that pressure waves, that oscillates in the normal direction at normal
incidence, are treated with the correct speed. The problem with this modification is that
since $\alpha_N = \sqrt{2+\lambda/\mu}$, $\alpha_N$ now depends on $\lambda$. In this setting
we are still able to provide a sharp estimate in terms of the frequency scaling below, but are
not able to establish any robustness as $\lambda/\mu \to +\infty$.

Although this is an unpleasing outcome, we propose an intuitive explanation as to why the
dependency in $\lambda$ arises here. It is well-known that the absorbing boundary condition
on $\GDiss = \partial B_\RRR$ is only a good approximation of the radiation condition when
$\RRR$ is large
\cite{bayliss_gunzburger_turkel_1982a,galkowski_lafontaine_spence_2021a,goldstein_1981a}.
However, the distance between the center of the domain and the absorbing boundary scales as
\begin{equation*}
\ks
=
\sqrt{\frac{\rho}{\mu}}
\omega\RRR
\qquad
\kp
=
\sqrt{\frac{\rho}{2\mu+\lambda}}
\omega\RRR
\end{equation*}
in terms numbers of pressure and shear wavelengths. In particular,
the relative distance measured in number of pressure wavelengths
tends to zero as $\lambda$ increases, meaning that the absorbing boundary
condition is not meant to be an accurate approximation of the radiation
condition in this case.

We now state a stability estimates with realistic Robin coefficients.
As we remarked above, the estimate is not robust in $\lambda/\mu$, but
the scaling in the frequency is optimal.

\begin{corollary}[Star-shaped obstacle with realistic Robin coefficients]
\label{corollary_obstacle_realistic}
Assume that $\alpha_N = \sqrt{2+\lambda/\mu}$ and $\alpha_T = 1$. Then,
\begin{equation}
\label{eq_obstacle_realistic}
\omega^2\|\uuu\|_{\rho,\Omega}
\leq
\left \{
\frac{1}{2}
+
\ks^{1/6}
+
\frac{1}{4} \ks^{1/3}
+
\left ( \frac{7}{2} + \frac{1}{2} \left (\frac{\lambda}{\mu} \right )^{1/2}\right ) \ks^{2/3}
+
\left ( 2 + \frac{1}{4}\frac{\lambda}{\mu} \right )\ks
\right \}\|\fff\|_{\rho,\Omega}.
\end{equation}
\end{corollary}

\begin{proof}
The estimate in \eqref{eq_obstacle_realistic} follows from
\eqref{eq_simplified_morawetz_estimate_star_shaped} upon observing
that 
\begin{equation*}
\Crob
=
\left (
2 + \sqrt{\alpha_{\max}}
\right )
\sqrt{\alpha_{\max}}
\leq
2\alpha_{\max} = 2 \left (2 + \frac{\lambda}{\mu}\right )^{1/2}.
\end{equation*}
for our choice of Robin coefficients.
% In this case, we end up with
% \begin{align*}
% \nonumber
% \omega^2\|\uuu\|_{\rho,\Omega}
% \leq
% \Big \{
% &
% \left (\frac{d}{2}-1\right )
% +
% \ks^{1/6}
% +
% \frac{1}{4} \ks^{1/3}
% +
% \\
% \nonumber
% &
% \left ( \frac{5}{2} + \frac{1}{2}\left (2 + \frac{\lambda}{\mu} \right )^{1/2} \right ) \ks^{2/3}
% +
% \\
% &
% \left (
% \frac{3}{2} + \frac{1}{4} \left (2 + \frac{\lambda}{\mu} \right )
% \right )\ks
% \Big \}\|\fff\|_{\rho,\Omega}
% \end{align*}
% which we can further simpily as
% \begin{align*}
% \omega^2\|\uuu\|_{\rho,\Omega}
% \leq
% \left \{
% \left (\frac{d}{2}-1\right )
% +
% \ks^{1/6}
% +
% \frac{1}{4} \ks^{1/3}
% +
% \left ( \frac{7}{2} + \frac{1}{2} \left (\frac{\lambda}{\mu} \right )^{1/2}\right ) \ks^{2/3}
% +
% \left ( 2 + \frac{1}{4}\frac{\lambda}{\mu} \right )\ks
% \right \}\|\fff\|_{\rho,\Omega}.
% \end{align*}
\end{proof}

\subsection{Smooth heterogeneous media}
\label{section_smooth_coefficients}

We now turn our attention to heterogeneous media where $\rho,\mu$ and $\lambda$
depend on the space variable. For simplicity, we restrict our attention to the
case where $\Omega = B_\RRR$, although the case $\Omega = B_\RRR \setminus D$
with $D \subset B_\RRR$ star-shaped poses no additional difficulty. We consider
coefficients $\rho,\mu,\lambda \in C^{1,1}(\overline{\Omega})$ with radial growth
conditions. Specifically, we require that $\rho$ is not ``too decreasing'' in the
sense that
\begin{equation*}
\bx \cdot \grad \rho \geq -\Theta_\rho \rho
\end{equation*}
for some $\Theta_\rho \leq 2$, and conversely, we demand that
$\mu$ and $\lambda$ are not ``too increasing'', by which mean that
\begin{equation*}
\bx \cdot \grad \mu \leq \Theta_\mu \mu,
\qquad
\bx \cdot \grad \lambda \leq \Theta_\lambda \lambda,
\end{equation*}
for $\Theta_\mu,\Theta_\lambda \leq 2$. We will in fact further constrain
these values by requiring that
\begin{equation*}
\Theta \eq \Theta_\rho+\max(\Theta_\mu,\Theta_\lambda) < 2.
\end{equation*}
Recalling their definition in \eqref{eq_definition_velV},
the above requirements essentially says that the shear and pressure
wavespeeds cannot increase to rapidly when moving away from the origin.
Similar constraints on the wave speed also naturally arise for scalar Helmholtz
problems \cite{barucq_chaumontfrelet_gout_2017a,moiola_spence_2019a} and Maxwell's
equations \cite{chaumontfrelet_moiola_spence_2023a,verfurth_2019a}.

The assumptions discussed above ensure that
\begin{equation*}
\LV_{\bh}(\rho) \geq -\Theta_\rho,
\qquad \LV_{\bh}(\mu) \leq \Theta_\mu,
\qquad
\LV_{\bh}(\lambda) \leq \Theta_\lambda,
\end{equation*}
so that, recalling \eqref{eq_LR_Omega_bh_bx},
\begin{equation*}
\LR_{\bh,\Omega}^2(\bv)
\leq
(d-2+\Theta_{\mu}) 2\|\eps(\bv)\|_{\mu,\Omega}^2
+
(d-2+\Theta_{\lambda}) \|\div \bv\|_{\lambda,\Omega}^2.
\end{equation*}
and
\begin{equation*}
\int_\Omega (\div \bh+\LV_{\bh}(\rho)) \rho|\bv|^2
\geq
(d-\Theta_\rho)\|\bv\|_{\rho,\Omega}^2.
\end{equation*}
It follows that \eqref{eq_assumption_gamma} holds true with
\begin{equation*}
2\gamma \eq 2-\Theta > 0.
\end{equation*}

\begin{corollary}[Smooth heterogeneous media]
Under the assumptions stated above, the estimate in \eqref{eq_stability_estimate} holds true.
\end{corollary}

\subsection{Transmission problems}

In this section, our goal is to relax the smoothness
assumption on the coefficients to deal with transmission
problems. We consider the case where $\rho$ is radially increasing and $\mu$ and $\lambda$
are decreasing radially. Specifically, we assume that $\rho,\mu,\lambda \in L^\infty(\Omega)$
\begin{equation}
\label{eq_monotonicity}
\rho((1+h)\bx)-\rho(\bx) \geq 0,
\quad
\mu((1+h)\bx)-\mu(\bx) \leq 0,
\quad
\lambda((1+h)\bx)-\lambda(\bx) \leq 0,
\end{equation}
for a.e. $\bx \in \Omega$ and $h > 0$ s.t. $(1+h)\bx \in \Omega$.
Notice that then, if we additionally assume that $\rho$, $\mu$ and $\lambda$ are
in $C^{0,1}(\overline{\Omega})$, then the discussion in Section \ref{section_smooth_coefficients}
is valid with $\Theta = 0$ and, therefore, with $\gamma = 1$. In fact, following
\cite{chaumontfrelet_moiola_spence_2023a}, Theorem \ref{theorem_simple_robin}
remains valid under the sole assumption that \eqref{eq_monotonicity} holds true.
This is established by first mollifying the coefficients while preserving the
monotonicity condition (this can be done following
\cite[Lemma 3.4]{chaumontfrelet_moiola_spence_2023a}), and applying the stability results
for the mollified problem and finally remarking the solution to the mollify
problem converges towards the solution of the problem with original coefficients.
For the sake of simplicity, we do not reproduce the proof here and refer the reader
to \cite[Theorem 1.2]{chaumontfrelet_moiola_spence_2023a} for more details.

We now discuss relevant choice of Robin parameters. Having in mind finite
element discretization, we focus on the case where the coefficients are
constant in a neighborhood of $\GDiss$. Due to the monotonicity assumption
in \eqref{eq_monotonicity}, it means that $\rho = \rho_{\max}$, $\mu = \mu_{\min}$
and $\lambda = \lambda_{\min}$ in the vicinity of the dissipative boundary.
As explained above, the choice $\AAA = \sqrt{\rho_{\max}\mu_{\min}}$ (i.e.
$\alpha_T = \alpha_N = 1$) leads to very good stability estimates, but is not
suited to absorb pressure waves. In contrast, the choice $\alpha_T \eq 1$ and
$\alpha_N = \sqrt{2+\lambda_{\min}/\mu_{\min}}$ is realistic from an application
perspective.

\begin{corollary}[Transmission problem]
\label{corollary_transmission_problem}
Assume that the monotonicity assumption in \eqref{eq_monotonicity}
is satisfied. Then, if $\alpha_N = \alpha_T = 1$, the estimates in
\eqref{eq_obstacle_ideal} hold true. In addition, when
$\alpha_N = \sqrt{2+\lambda_{\min}/\mu_{\min}}$ and $\alpha_T = 1$,
the estimate in \eqref{eq_obstacle_realistic} holds true with
$\lambda/\mu$ replaced by $\lambda_{\min}/\mu_{\min}$.
\end{corollary}

Crucially, the estimates in Corollary \ref{corollary_transmission_problem}
are robust in $\lambda$ in the sense that the bounds remain uniform in
$\lambda_{\max}/\mu_{\min}$. In particular, consider the scattering problem
by a star-shaped penetrable obstacle $D \subset \Omega$ whereby
$\rho \equiv \rho_D$, $\mu \equiv \mu_D$ and $\lambda \equiv \lambda_D$
in $D$ and
$\rho \equiv \rho_0$, $\mu \equiv \mu_0$ and $\lambda \equiv \lambda_0$
outside $D$, the estimates in Corollary \ref{corollary_transmission_problem}
are valid as long as $\rho_D \leq \rho_0$, $\mu_D \geq \mu_0$ and $\lambda_D \geq \lambda_0$.

\subsection{Almost star-shaped obstacles}

In this section, we discuss the possibility of using multipliers $\bh$ different from $\bx$.
In particular, following \cite{martinez_1999a}, our goal will be to handle ``almost
star-shaped'' obstacles by consider multipliers $\bh = \bx + \grad \phi$,
where $\grad \phi$ is sufficiently small perturbation. For the sake of simplicity,
we assume here that $\rho,\mu$ and $\lambda$ are constant. We also assume that
$\GDir \neq \emptyset$. As we will see, unfortunately, such an approach seems only
possible for ``small'' values of $\lambda/\mu$. Due to the Dirichlet boundary
condition on $\GDir$, we have
\begin{equation}
\label{eq_korn_dirichlet}
\|\grad \bv\|_{\mu,\Omega}
\leq
\LK_0\|\eps(\bv)\|_{\mu,\Omega}
\end{equation}
for all $\bv \in \BH^1_{\GDir}(\Omega)$, where the constant $\LK_0 \geq 1$ depends
on the shape of $\Omega$ and $\GDir$. See, e.g. \cite[Theorem 11.2.12]{brenner_scott_2008a}.

\begin{theorem}
Assume that $\phi \in C^{2,1}(\overline{\Omega})$ with $\phi = 0$ in a neighborhood of
$\GDiss$. Assume that $\bh \cdot \bn \leq 0$ on $\GDir$. Assume that there exists
$\varepsilon,\eta \geq 0$ with $\varepsilon+\eta < 2$, and
\begin{equation*}
\left (1+2\LK_0 + \frac{\lambda}{\mu}\right ) \|\tgrad^2\phi\|_{\tens{L}^\infty(\Omega)}
\leq
\varepsilon
\qquad
\min_{\Omega} \Delta \phi \geq -\eta.
\end{equation*}
Then estimate \eqref{theorem_simple_robin} from Theorem \ref{eq_stability_estimate} holds true.
\end{theorem}

\begin{proof}
Following the assumptions given in Section \ref{section_key_assumptions},
we need to check that
\begin{equation}
\label{tmp_ass_L2}
(d-\eta)\|\bv\|_{\rho,\Omega}^2
\leq
\int_\Omega (\div \bh) \rho |\bv|^2
\end{equation}
and
\begin{equation}
\label{tmp_ass_H1}
|\LR_{\bh,\Omega}(\bv)| \leq (d-2+\varepsilon) 
\{
2\|\eps(\bv)\|_{\mu,\Omega}^2 + \|\div \bv\|_{\lambda,\Omega}^2
\}
\end{equation}
for all $\bv \in \BH^1_{\GDir}(\Omega)$.

We easily see that \eqref{tmp_ass_L2} is valid by noting that
\begin{equation*}
\div \bh = \div \bx + \div \grad \phi = d + \Delta \phi \geq d-\eta
\end{equation*}
by assumption.

To establish \eqref{tmp_ass_H1}, we first follow the computations
performed for $\bh = \bx$ in Section \ref{section_simple_robin} to see that
\begin{align*}
\LR_{\bh,\Omega}(\bv)
&\eq
\int_\Omega \left \{
(d-2 + \Delta \phi)2\mu|\eps(\bv)|^2
+
(d-2 + \Delta \phi)\lambda|\div \bv|^2
\right \}
\\
&-2\Re\int_\Omega \left \{
2 \mu \eps(\bv) : (\tgrad^2\phi \tgrad \overline{\bv})
+
\lambda (\div \bv) \tgrad^2\phi : (\tgrad \overline{\bv})^T
\right \}
\end{align*}
Then, since $\tgrad^2\phi$ is symmetric, we have
\begin{equation*}
|\tgrad^2 \phi:(\tgrad \overline{\bv})^T|
=
2|\tgrad^2 \phi:\eps(\overline{\bv})|
\leq
2\|\tgrad^2{\phi}\|_{\tens{L}^\infty(\Omega)}|\eps(\bv)|,
\end{equation*}
and therefore
\begin{equation*}
2 \Re \lambda (\div \bv) \tgrad^2 \phi : (\tgrad \overline{\bv})^T
\leq
4\|\tgrad^2{\phi}\|_{\tens{L}^\infty(\Omega)}\lambda|\div \bv||\eps(\bv)|
\leq
2\|\tgrad^2{\phi}\|_{\tens{L}^\infty(\Omega)}\left (\lambda|\div \bv|^2
+
\frac{\lambda}{\mu} \mu |\eps(\bv)|^2\right )
\end{equation*}
On the other hand, since
\begin{equation*}
|\tgrad^2\phi\tgrad \overline{\bv}|
\leq
\|\tgrad^2{\phi}\|_{\tens{L}^\infty(\Omega)}|\tgrad \overline{\bv}|,
\end{equation*}
we obtain
\begin{equation*}
2 \Re \int_\Omega 2\mu\eps(\bv):\tgrad^2\phi\tgrad \overline{\bv}
\leq
4\|\tgrad^2{\phi}\|_{\tens{L}^\infty(\Omega)}\|\eps(\bv)\|_{\mu,\Omega}\|\tgrad \bv\|_{\mu,\Omega}
\leq
2\LK_0\|\tgrad^2{\phi}\|_{\tens{L}^\infty(\Omega)}\cdot2\|\eps(\bv)\|_{\mu,\Omega}^2,
\end{equation*}
where $\LK_0$ is the Korn constant from \eqref{eq_korn_dirichlet}.
It follows that
\begin{align*}
\LR_{\bh,\Omega}(\bv)
&\leq
\left (d-2+\left (1+2\LK_0+\frac{\lambda}{\mu}\right )\|\tgrad^2{\phi}\|_{\tens{L}^\infty(\Omega)}\right )
2\|\eps(\bv)\|_{\mu,\Omega}^2
\\
&+
\left (d-2+3\|\tgrad^2{\phi}\|_{\tens{L}^\infty(\Omega)}\right )
\|\div \bv\|_{\lambda,\Omega}^2,
\end{align*}
and the conclusion follows since $3 \leq 1+2\LK_0 \leq 1 + 2\LK_0 + \lambda/\mu$.
\end{proof}

\subsection{General Robin boundary}

In all the examples given above, the estimate of Theorem \ref{theorem_general_robin}
given in \eqref{eq_general_robin} holds true if relax the assumption that $\AAA$
satisfies \eqref{eq_assumption_AAA} and only assume that $\GDiss$ is the boundary
of a closed set star-shaped with respect to $\bzero$ and that Assumption
\ref{assumption_elliptic_regularity} is satisfied. The resulting estimate is
robust in $\lambda_{\max}/\mu_{\min}$ if $\alpha_{\max}$ and $\alpha_{\min}$
do not depend on $\lambda$ and $\mu$.

%% file: appendix.tex
\appendix

\section{Stability estimates using the Green's function}
\label{appendix_fundamental_solution}

In this section, we assume that $\rho$, $\mu$ and $\lambda$
are constant real numbers. Recall then the (constant) shear and
pressure wavespeeds are defined in \eqref{eq_definition_vels} as
\begin{equation*}
\velsV \eq \sqrt{\frac{\mu}{\rho}}
\qquad
\velpV \eq \sqrt{\frac{2\mu+\lambda}{\rho}}
\end{equation*}
with the corresponding wavenumbers given by
$\Ks \eq \omega/\velsV$ and $\Kp \eq \omega/\velpV$.

The Green's function associated with the PDE
\begin{equation}
\label{eq_appendix_elasticity}
-\omega^2 \rho \uuu-\div \sig(\uuu) = \rho \fff \text{ in } \mathbb R^3
\end{equation}
completed by the Sommerfeld radiation condition at infinity
is given, e.g., in \cite[Equation (2.10)]{brown_gallistl_2023a}.
Specifically, we have
\begin{equation*}
\uuu = \tens{G} \star \fff
\end{equation*}
with
\begin{equation}
\label{eq_green_elasticity}
\tens{G}(\by)
\eq
\frac{1}{4\pi\velsV^2}
\left (
\frac{e^{i\Ks |\by|}}{|\by|}
\tens{I}
+
\frac{1}{\Ks^2}
\tgrad^2
\left (
\frac{e^{i\Ks |\by|}}{|\by|} - \frac{e^{i\Kp |\by|}}{|\by|}
\right )
\right ).
\end{equation}
Notice that as compared to \cite{brown_gallistl_2023a},
our wavenumbers $\Ks$ and $\Kp$ (these are respectively called
$k_1$ and $k_2$ in \cite{brown_gallistl_2023a}) are multiplied
by $\sqrt{\rho}$, which is due to the fact that the case $\rho = 1$
is considered in \cite{brown_gallistl_2023a}.
This is fully consistant, as can be seen by dividing
both sides of \eqref{eq_appendix_elasticity} by $\rho$
and redefining $\mu$ and $\lambda$ appropriately.

It is convenient to introduce the notation
\begin{equation*}
G^{\rm A}(\by)
\eq
\frac{1}{4\pi\velsV^2} \frac{e^{i\Ks |\by|}}{|\by|},
\qquad
G^{\rm E}(\by)
\eq
\frac{1}{4\pi} 
\left (
\frac{e^{i\Ks |\by|}}{|\by|}
-
\frac{e^{i\Kp |\by|}}{|\by|},
\right ),
\end{equation*}
in order to write
\begin{equation}
\label{eq_green_split}
\tens{G}
\eq
G^{\rm A} \tens{I} + \frac{1}{\omega^2} \tgrad^2 G^{\rm E}.
\end{equation}
Crucially, $G^{\rm A}$ is the Green's function associated with the acoustic
(scalar) Helmholtz equation. In particular \cite{galkowski_spence_wunsch_2020a},
for $\RRR > 0$ and $f \in L^2(\R^3)$ with $\supp f \subset B_\RRR$, we have
\begin{equation}
\label{eq_stability_green_scalar}
\omega^2 \|G^{\rm A} \star f\|_{\rho,B_\RRR}
\leq
\ks \|f\|_{\rho,B_\RRR},
\end{equation}
with, as above, $\ks \eq \Ks\ell$. The other part of $\tens{G}$ requires
a bit of extra work. Our proof closely follows \cite[Lemma A.2]{brown_gallistl_2023a},
but we take special care to track the dependency on $\lambda/\mu$. Our improvements
over \cite{brown_gallistl_2023a} are very minor and far from complicated, but for
the sake of completeness, we include them here.

\begin{lemma}[Elastic part of  Green's function]
\label{lemma_green_elasticity}
Let $\RRR > 0$ and $f \in L^2(\R^3)$ with $\supp f \subset B_\RRR$. Then, we have
\begin{equation}
\label{eq_green_elsaticity}
\|\tgrad^2(G^{\rm E} \star f)\|_{B_\RRR}
\leq
(2+8\Ks\RRR) \|f\|_{B_\RRR}.
\end{equation}
\end{lemma}

\begin{proof}
Here, for $v \in L^2(\R^3)$, we let
\begin{equation*}
\LF(v)(\bxi) = \int_{\R^3} v(\bx) e^{i\bxi\cdot\bx}d\bx
\end{equation*}
denote the Fourier transform of $v$. With this convention, we have
\begin{equation*}
\|v\|_{\R^3}^2 = \frac{1}{(2\pi)^3} \|\LF(v)\|_{R^3}^2.
\end{equation*}
and
\begin{equation*}
\LF(\phi \star v) = \LF(\phi)\cdot\LF(v)
\end{equation*}
for all $\phi \in L^2(\R^3)$. The notation
\begin{equation*}
g(r)
\eq
\frac{1}{4\pi}
\left (
\frac{e^{i \Ks r}}{r}-\frac{e^{i \Kp r}}{r}
\right )
\end{equation*}
will also be useful to write $G^{\rm E}(\by) = g(|\by|)$.

The proof relies on a modified Green's functions $M \cdot G^{\rm E}$,
where $M(\by) = \eta(|\by|)$, with a cutoff function $\eta$ defined as follows.
If $r \leq 2\RRR$ we let $\eta(r) = 1$ where as we set $\eta(r) = 0$ if $r \geq 4\RRR$.
The cutoff changes linearly in the transition region, namely
$\eta(r) = (4\RRR-r)/(2\RRR)$ for $2\RRR < r < 4\RRR$.
$\eta$ is Lipschitz continuous, $|\eta(r)| \leq 1$
and $|\eta'(r)| \leq 1/(2\RRR)$ for all $r \geq 0$.

Let us consider the modified solution $\tw \eq (M \cdot G^{\rm E}) \star f$.
Since $f$ is supported  in $B_\ell$, if $\bx \in B_\ell$, we have
\begin{equation*}
\tw(\bx)
=
\int_{B_\ell} \eta(|\bx-\by|) g(|\bx-\by|) f(\by) d\by
=
\int_{B_\ell} g(|\bx-\by|) f(\by) d\by
=
(G^{\rm E} \star f)(\bx)
\end{equation*}
as $|\bx-\by| \leq |\bx|+|\by| \leq 2\ell$ whenever $\bx,\by \in B_\ell$. It follows that
\begin{equation*}
\|\tgrad^2(G^{\rm E} \star f)\|_{B_\ell}
=
\|\tgrad^2 \tw\|_{B_\ell}
\leq
\|\tgrad^2 \tw\|_{\R^3}
=
\|\Delta \tw\|_{\R^3}.
\end{equation*}
where the last identity follows from two integrations by parts.
We estimate the last norm using the Fourier transform. Indeed, since
\begin{equation*}
\LF(\tw) = 
|\bxi|^2 \LF(\tw) = |\bxi|^2 \LF(M \cdot G^{\rm E}) \LF(f),
\end{equation*}
we infer that
\begin{equation*}
%\|\Delta \tw\|_{\R^3}
\| |\bxi|^2 \LF(\tw)\|_{\R^3}^2
\leq
\LM \|f\|_{B_\ell},
\qquad
\LM
\eq
\sup_{\bxi \in \R^d} |\bxi|^2|\LF(M \cdot G^{\rm E})(\bxi)|,
\end{equation*}
and it remains to estimate $\LM$. Since $(M \cdot G^{\rm E})(\by) = (\eta\cdot g)(|\by|)$,
we can classically express its Fourier transform as follows
\begin{equation*}
\LF(M \cdot G^{\rm E})(\bxi)
=
4\pi \int_0^{+\infty}
\eta(r)g(r) r^2 \frac{\sin(|\bxi|r)}{|\bxi|r} dr.
\end{equation*}
The Fourier transform further simplifies into
\begin{align*}
|\bxi|^2\LF(M \cdot G^{\rm E})(\bxi)
&=
\int_0^{+\infty}
\eta(r)(e^{i\Ks r}-e^{i\Kp r}) |\bxi| \sin(|\bxi|r) dr
\\
&=
\int_0^{+\infty}
\partial_r (\eta(r)(e^{i\Ks r}-e^{i\Kp r})) \cos(|\bxi|r) dr
\end{align*}
upon recalling the definion of $g$. Here, we can conclude the proof with
\begin{equation*}
\LM
\leq
\int_0^{4\RRR} |\partial_r (\eta(r)(e^{i\Ks r}-e^{i\Kp r}))|
\leq
4\RRR(\Ks + \Kp) + 2
\leq
2+8\Ks\RRR.
\end{equation*}
\end{proof}

\begin{theorem}[Stability]
Fix $\RRR > 0$ and $\fff \in \BL^2(\R^3)$ with $\supp \fff \subset B_\RRR$.
For $\uuu \eq \tens{G} \star \fff$, we have
\begin{equation*}
\omega^2\|\uuu\|_{\rho,B_\RRR} \leq \left (4 + 17\ks\right )\|\fff\|_{\rho,B_\RRR}.
\end{equation*}
\end{theorem}

\begin{proof}
Recalling \eqref{eq_green_split}, we may split $\uuu$ as
\begin{equation*}
\uuu = (G^{\rm A} \tens{I})\star \fff + \frac{1}{\omega^2} (\tgrad^2 G^{\rm E}) \star \fff.
\end{equation*}
On the one hand, each component of $\uuu^{\rm A}_\ell$ of 
$\uuu^{\rm A} \eq G^{\rm A}\tens{I} \star \fff$ independently solves
a scalar Helmholtz problem with wavenumber $\Ks$ and right-hand side $\fff_\ell$,
i.e. $\uuu^{\rm A}_\ell = G^{\rm A} \star \fff_\ell$, $1 \leq \ell \leq 3$.
Therefore, it follows from \eqref{eq_stability_green_scalar} that
\begin{equation*}
\omega^2 \|(G^{\rm A}\tens{I}) \star \fff\|_{\rho,B_\RRR} \leq \ks \|\fff\|_{\rho,B_\RRR}.
\end{equation*}

For the other part, we rely on Lemma \ref{lemma_green_elasticity}.
Let us set $\uuu^{\rm E} \eq (\tgrad^2 G^{\rm E}) \star \fff$, so that
\begin{equation*}
\uuu^{\rm E}_j
=
\sum_{\ell=1}^3(\partial_{j\ell} G^{\rm E}) \star \fff_\ell
=
\sum_{\ell=1}^3 \partial_{j\ell} (G^{\rm E} \star \fff_\ell),
\end{equation*}
and
\begin{multline*}
\|(\tgrad^2 G^{\rm E})\star\fff\|_{B_\RRR}^2
=
\sum_{j=1}^3 \|\uuu^{\rm E}_j\|_{B_\RRR}^2
\leq
\sum_{j=1}^3
\left \|
\sum_{\ell=1}^3 \partial_{j\ell} (G^{\rm E} \star \fff_\ell)
\right \|_{B_\RRR}^2
\\
\leq
3
\sum_{j,\ell=1}^3
\|\partial_{j\ell} (G^{\rm E} \star \fff_\ell)\|_{B_\RRR}^2
=
3
\|\tgrad^2(G^{\rm E} \star \fff_\ell)\|_{B_\RRR}^2
\leq
4
\|\tgrad^2(G^{\rm E} \star \fff_\ell)\|_{B_\RRR}^2.
\end{multline*}
Since $\rho$ is constant, it follows from \eqref{eq_green_elasticity} that
\begin{equation*}
\|(\tgrad^2 G^{\rm E})\star\fff\|_{\rho,B_\RRR}
\leq
2
\|(\tgrad^2 G^{\rm E}) \star \fff\|_{\rho,B_\RRR}
\leq
(4 + 16\Ks\RRR)\|\fff\|_{\rho,B_\RRR},
\end{equation*}
and we easily condlude the proof since
\begin{equation*}
\omega^2\|\uuu\|_{\rho,B_\RRR}
\leq
\omega^2\|(G^{\rm A}\tens{I})\star \fff\|_{\rho,B_\RRR}
+
\|(\tgrad^2G^{\rm E})\star \fff\|_{\rho,B_\RRR}.
\end{equation*}
\end{proof}

\section{Elliptic regularity for the elastostatic problem}
\label{appendix_elliptic_regularity}

The goal of this section is to show that Assumption
\ref{assumption_elliptic_regularity} does hold true
in relevant situations.

\subsection{Validity under elliptic regularity}

We start by showing that Assumption \ref{assumption_elliptic_regularity}
holds true under the more natural elliptic regularity assumption on
the elastostatic problem below.

\begin{assumption}[Elliptic regularity]
\label{assumption_natural}
Assume that $\bw \in \BH^1_{\GDir}(\Omega)$ is such that
\begin{equation}
\label{eq_elastostatic_appendix}
(\sig(\bw),\eps(\bv))_\Omega = \omega^2 (\rho\bg,\bv)_\Omega + \omega(\AAA\bk,\bv)_{\GDiss}
\end{equation}
for all $\bv \in \BH^1_{\GDir}(\Omega)$ with $\bg \in \BL^2(\Omega)$
and $\bk \in \BH^1_{\GDir}(\Omega)$. Then there exist a neighborhood
$\widetilde \Omega \subset \Omega$ of $\GDiss$ such that, $\bw \in \BH^2(\widetilde \Omega)$
and
\begin{equation}
\label{eq_regularity_estimate}
\mu_{\min}^{1/2}|\bw|_{\BH^2(\widetilde \Omega)}
\leq
\Cell
\frac{1+\ks}{\ell}
\left (
\omega \|\bg\|_{\rho,\Omega} + \omega \|\bk\|_{\rho,\Omega} + \|\eps(\bk)\|_{\mu,\Omega}
\right ),
\end{equation}
for some $\Cell > 0$.
\end{assumption}

\begin{lemma}[Validity of the Assumption \ref{assumption_elliptic_regularity}]
If Assumption \ref{assumption_natural} is valid, then Assumption
\ref{assumption_elliptic_regularity} holds true with
\begin{equation*}
\Creg \leq C (1+\Cell),
\end{equation*}
where $C$ only depends on $\Omega$ and on the neighborhood $\widetilde \Omega$
in Assumption \ref{assumption_natural}. Specifically, if $\bw \in \BH^1_{\GDir}(\Omega)$
solves \eqref{eq_elastostatic_appendix}, then, $\grad \bw \in \tens{L}^2(\GDiss)$, and we have
\begin{equation}
\label{eq_assumption_verified}
\ell \|\grad \bw\|_{\mu,\GDiss}^2
\leq C (1+\Cell)
(1+\ks)
\left \{
\omega^2\|\bg\|_{\rho,\Omega}
+
\omega^2\|\bk\|_{\rho,\Omega} + \|\eps(\bk)\|_{\mu,\Omega}^2
+
\|\eps(\bw)\|_{\mu,\Omega}^2
\right \}
\end{equation}
\end{lemma}

\begin{proof}
First, using the  multiplicative trace inequality in $\hCO$ derived,
e.g., in \cite[Theorem 1.5.1.10]{grisvard_1985a} and using a scaling argument, we get
\begin{equation*}
\RRR \|\grad \bw\|_{\GDiss}^2
\leq C
\left (
\|\grad \bw\|_{\widetilde \Omega}^2
+
\RRR \|\grad \bw\|_{\widetilde \Omega}|\bw|_{H^2(\widetilde \Omega)}
\right )
\leq C
\left (
(1+\ks) \|\grad \bw\|_\Omega^2 + \frac{\RRR^2}{\ks} |\bw|_{H^2(\widetilde \Omega)}^2
\right ).
\end{equation*}
By Korn inequality, we then have
\begin{equation*}
\RRR \|\grad \bw\|_{\mu,\GDiss}^2
\leq C
\left (
(1+\ks) \|\eps(\bw)\|_{\mu,\Omega}^2 + \frac{\RRR^2}{\ks} \mu_{\min} |\bw|_{H^2(\widetilde \Omega)}^2
\right )
\end{equation*}
and \eqref{eq_assumption_verified} follows from \eqref{eq_regularity_estimate}.
\end{proof}

\subsection{Elliptic regularity}

We finally state two standard elliptic regularity results for the
elastostatic system that are easily found in the litterature.

\begin{proposition}[Locally smooth cofficients]
\label{lemma_elliptic_regularity}
Assume that $\AAA \in \tens{C}^{1,1}(\overline{\GDiss})$ and that there
exists a neighborhood $U \subset \Omega$ of $\GDiss$ such that
$\mu,\lambda$ are of class $C^{0,1}$. Then, \eqref{eq_elastostatic_appendix}
holds true with $U = \widetilde \Omega$ and a constant $\Cell$
depending on $\Omega$ and $\widetilde \Omega$, the coefficients $\lambda,\mu,\rho$,
$\AAA$ and the ratio $\lambda_{\max}/\mu_{\min}$.
\end{proposition}

\begin{proof}
The result follows from standard elliptic regularity for the elastostatic system,
see e.g. \cite{agmon:59} or \cite[Theorem 2.3.2]{GLCI}.
\end{proof}

\begin{proposition}[Constant coefficients]
\label{lemma_elliptic_regularity}
Assume that that $\GDiss$ and $\GDir$ are of class $C^{1,1}$.
We also assume that $\AAA \in \tens{C}^{1,1}(\overline{\GDiss})$,
and that $\mu,\lambda$ are constant. Then, \eqref{eq_elastostatic_appendix}
holds true with $\Omega = \widetilde \Omega$ and a constant $\Cell$
independent of $\lambda/\mu$.
\end{proposition}

\begin{proof}
See \cite[Lemma A.1]{vogelius_1983a}.
\end{proof}